\documentclass[10pt]{article}
\usepackage[utf8]{inputenc}
\usepackage[backend=biber]{biblatex}
\bibliography{./ref}

\usepackage{appendix}
\usepackage{standalone}
\usepackage{geometry}
\usepackage{setspace}
\usepackage{amssymb}
\usepackage{mathtools}
\usepackage{amsfonts}
\usepackage{amsmath}
\usepackage{amsthm}
\usepackage{eufrak}
\usepackage{arydshln}
\usepackage{bm} 
\usepackage{mathtools}
\usepackage{tikz-cd}
\usepackage{ytableau}
\usepackage{mathrsfs}
\usepackage{adjustbox}
\usepackage{qtree}
\usepackage{pstricks-add}
\usepackage{tcolorbox}
\usepackage{hyperref}
\hypersetup{
    colorlinks=true,
    citecolor=red,
    linkcolor=blue,
    filecolor=magenta,      
    urlcolor=cyan,
}

\theoremstyle{plain}

\theoremstyle{definition}
\newtheorem{thm}{Theorem}[section] 
\newtheorem*{thm*}{Theorem}
\newtheorem{defn}[thm]{Definition} 
\newtheorem{eg}[thm]{Example} 
\newtheorem{lma}[thm]{Lemma} 
\newtheorem{propn}[thm]{Proposition} 
\newtheorem{remq}[thm]{Remark}
\newtheorem{cor}[thm]{Corollary}

\newtheorem{assumption}[thm]{Assumption}

\newtheorem*{problem}{Problem}

\newcommand\restr[2]{{
		\left.\kern-\nulldelimiterspace 
		#1 
		\vphantom{\big|} 
		\right|_{#2} 
}}

\newcommand{\Ci}{\mathbb{C}}
\newcommand{\Z}{\mathbb{Z}}
\newcommand{\N}{\mathbb{N}}
\newcommand{\R}{\mathbb{R}}
\newcommand{\Pe}{\mathbb{P}}
\newcommand\Set[2]{\{\,#1\mid#2\,\}}
\newcommand{\vi}{\text{v}}
\newcommand{\w}{\text{w}}

\newcommand{\Stab}[2]{\text{Stab}_\mathfrak{#1}(#2)}

\newcommand{\Base}[2]{\mathscr{B}_{#1,#2}}
\newcommand{\Repn}[3]{\text{Rep}_{#1}({#2},{#3})}
\newcommand{\pt}{\ast }

\newcommand{\Epic}[2]{\mathscr{E}_{\text{Pic}_{#1}(#2)}}

\newcommand{\G}{\text{G}}
\newcommand{\Gr}[2]{\text{Gr}({#1},{#2})}
\newcommand{\rk}{\text{rk}}
\newcommand{\ind}{\text{ind}}
\newcommand{\Pic}[2]{\text{Pic}_{#1}(#2)}
\newcommand{\Lie}{\text{Lie}}

\newcommand{\Hom}{\text{Hom}}

\newcommand{\V}{\mathscr{V}}
\newcommand{\aV}{\widetilde{\mathscr{V}}}
\newcommand{\W}{\mathscr{W}}
\newcommand{\aW}{\widetilde{\mathscr{W}}}
\newcommand{\Nbund}{\mathfrak{N}}

\newcommand{\n}{\mathfrak{n}}
\newcommand{\g}{\mathfrak{g}}

\newcommand{\GL}{\text{GL}}
\newcommand{\naka}{\mathfrak{M}}

\newcommand{\Att}[2]{\text{Att}_{\mathfrak{#1}}({#2})}

\newcommand{\Fl}{\text{Fl}}
\newcommand{\Shuffle}{\text{Shuffle}}

\newcommand{\Char}{\text{Char}}
\newcommand{\Hilb}{\text{Hilb}}

\newcount\dotcnt\newdimen\deltay
\def\Ddot#1#2(#3,#4,#5,#6){\deltay=#6\setbox1=\hbox to0pt{\smash{\dotcnt=1
\kern#3\loop\raise\dotcnt\deltay\hbox to0pt{\hss#2}\kern#5\ifnum\dotcnt<#1
\advance\dotcnt 1\repeat}\hss}\setbox2=\vtop{\box1}\ht2=#4\box2}

\newcommand{\Addresses}{{
  \bigskip
  \footnotesize

  \textsc{Tommaso Maria~Botta, Departement Mathematik, ETH Zürich,
    8092 Zürich, Switzerland}\par\nopagebreak
  \textit{E-mail address}: \texttt{tommaso.botta@math.ethz.ch}

}}

\title{Shuffle products for elliptic stable envelopes of Nakajima varieties}
\author{Tommaso Maria Botta}

\date{}
\begin{document}
\maketitle


\begin{abstract}
    We find an explicit formula that produces inductively the elliptic stable envelopes of an arbitrary Nakajima variety associated to a quiver $Q$ from the ones of those Nakajima varieties whose framing vectors are the fundamental vectors of the quiver $Q$, i.e. the dimension vectors with just one unitary nonzero entry. The result relies on abelianization of stable envelopes. As an application, we combine our result with Smirnov's formula for the elliptic stable envelopes of the Hilbert scheme of points on the plane to produce the elliptic stable envelopes of the instanton moduli space. 
\end{abstract}




\section{Introduction}

\subsection{Overview}

Let $T$ be a complex torus and $X$ a quasi-projective symplectic $T$-variety such that the torus $T$ rescales the symplectic form with weight $h^{-1}$. Set $A=\ker(h)$. 
Stable envelopes of $X$ \cite{maulik2012quantum} \cite{okounkov2020inductiveI} \cite{okounkov2020inductiveII} are certain uniquely defined correspondences in $X^A\times X$ that induce maps from the equivariant cohomology of the fixed locus $X^A$ to the equivariant cohomology of $X$. Originally introduced in \cite{maulik2012quantum} as Lagrangian cycles in equivariant singular cohomology, they have been more recently defined in K-theory and elliptic cohomology \cite{Ganter_2014}\cite{ginzburg1995elliptic}\cite{Grojnowski2007EllipticCD}. 
Stable envelopes of Nakajima quiver varieties play a privileged role in the theory because they have proven to be the key tools to link the rich geometry of these varieties with the realms of quantum groups and $q$-difference equations.  

Cohomological stable envelopes of Nakajima varieties have been used to associate to every quiver $Q$ a Hopf algebra $Y(\g_Q)$, called Yangian, acting on the cohomology of these varieties \cite{maulik2012quantum}. Similarly, the K-theoretic lift of stable envelopes has been used to produce actions of the quantum loop algebras $\mathscr{U}(\hat\g_Q)$, in terms of which the $q$-difference equations for Nakajima varieties have been determined \cite{okounkovsmirnov2016Kthy}, see also \cite{aganagic2018quantum}\cite{aganagic2016elliptic} \cite{okounkov2017enumerative}\cite{okounkov2020inductiveII} for further developments.
In all these results, the key role of stable envelopes is to build automorphisms of the cohomology ring of the fixed locus $X^A$ satisfying the Yang-Baxter equation
\[
R^{(12)}(u)R^{(13)}(u+v)R^{(23)}(v)=R^{(23)}(v)R^{(13)}(u+v)R^{(12)}(u).
\]
More recently in \cite{aganagic2016elliptic} \cite{okounkov2020inductiveI}\cite{okounkov2020inductiveII}, these ideas were generalized to a third generalized cohomology theory, namely elliptic cohomology. The associated elliptic stable envelopes allow to produce solutions $R(u,z)$ of the dynamical Yang Baxter equation,
originally introduced in \cite{felder1994conformal} and \cite{felder1994ellipticquantumgroups}, which depend on additional parameters $z$ living in a complex abelian Lie algebra $\mathfrak{h}$, called Kähler parameter.

Recall that given an elliptic curve $E=\Ci^\times/q^{\Z}$, the 0-th elliptic cohomology functor
\[
E_T(-):T\text{-spaces}\to \text{Schemes}_{E_T}
\]
is a functor from the category of $T$-equivariant topological spaces to schemes over $E_T:=E_T(\pt)=E\otimes_\Z\Char(T)$.
In the original presentation in \cite{aganagic2016elliptic}, the authors defined the elliptic stable envelopes $\text{Stab}_{\mathfrak{C}}$ of a Nakajima variety $X$ as morphisms of certain sheaves over $E_T\times \Epic{T}{X}$, where $\Epic{T}{X}$ is the abelian variety $\Pic{T}{X}\otimes_{\Z}E$, which is the geometric analogue of the Cartan algebra $\mathfrak{h}$, and should be thought as the space where the variables $z$ live in the geometric setting. The aforementioned sheaves arise as pushforwards of some line bundles over $E_T(X^A)\times \Epic{T}{X^A}$ and $E_T(X)\times \Epic{T}{X}$, on the domain and on the codomain respectively. 
Therefore, following the general philosophy of stable envelopes, they might be thought as maps from the cohomolgy of the fixed locus $X^A$ to the cohomology of $X$.

Existence of stable envelopes for Nakajima varieties was proved in \cite{aganagic2016elliptic} in two steps. In the first one the authors explicitly built stable envelopes for smooth hypertoric varieties. In the second one they associated to every Nakajima variety $X$ an hypertoric variety $X_S$ and showed that stable envelopes of $X$ can be obtained from the ones of $X_S$ via a chain of maps.
Recall that Nakajima varieties $\naka_{Q}(\vi,\w)$ are algebraic symplectic reductions of the space of linear representations $T^*\Repn{Q}{\vi}{\w}$ associated to a quiver $(I,Q)$ and dimension vectors $\vi,\w\in N^{I}$ with respect to the natural action of $G_{\vi}=\prod_{i\in I}\GL({\vi_i})$:
\[
\naka_{Q}(\vi,\w)=T^*\Repn{Q}{\vi}{\w}////G_{\vi}.
\]
The varieties $X_S$ are called \emph{abelianizations} of Nakajima varieties in the literature, and are realized as symplectic reductions of $T^*\Repn{Q}{\vi}{\w}$ with respect to the action of a maximal torus $S$ in $G_\vi$:
\[
T^*\Repn{Q}{\vi}{\w}////S.
\]
This construction suggests to investigate the relationship between stable envelopes of a Nakajima variety $X$ and the ones associated to the symplectic reduction of $T^*\Repn{Q}{\vi}{\w}$ with respect to the action of a connected algebraic group $H$ interlacing $S$ and $G_{\vi}$:
\[
S\subset H\subset G_\vi.
\]
In this article we study the case when $H\cong G_{\vi'}\times G_{\vi''}$, with $\vi'+\vi''=\vi$. After having observed that, in complete analogy to the case $H=S$ proved in \cite{aganagic2016elliptic}, stable envelopes of $X$ can be recovered from the ones of
\[
T^*\Repn{Q}{\vi}{\w}////G_{\vi'}\times G_{\vi''},
\]
we show that the latter can be built in turn from the stable envelopes of two ``smaller" Nakajima varieties 
\[
T^*\Repn{Q}{\vi'}{\w'}////G_{\vi'} \qquad T^*\Repn{Q}{\vi''}{\w''}////G_{\vi''}
\]
such that $\vi'+\vi''=\vi$ and $\w'+\w''=\w$.

Combining these results, we deduce that when $A=A_\w$, where $A_\w$ is a maximal torus of $G_\w$ acting on $\naka_Q(\vi,\w)$, then we have:
\begin{thm*}[\ref{shuffle theorem}]
Let $\mathfrak{C}$ be a chamber and $F$ a fixed component of $\naka_{Q}(\vi,\w)$. There exist decompositions $\w=\w'+\w''$ and $\vi=\vi'+ \vi''$ and fixed components $F^{\vi'}\subset \naka_{Q}(\vi',\w')$ and $F^{\vi''}\subset \naka_{Q}(\vi'',\w'')$ such that:
\begin{equation*}
    \Stab{C}{F}=\Shuffle\restr{\left\lbrace\frac{\Theta(\mathscr{N}_{\hat j })\tau^*(-h\delta)\left((\Stab{C'}{F^{\vi'}}\boxtimes \Stab{C''}{F^{\vi''}})^{lift}\right)\circ (\pi^{A}!)^{-1}}{\Theta(\Nbund^\vee)\Theta(h^{-1}\Nbund^\vee)}\right\rbrace}{z'=z''=z}.
\end{equation*}
\end{thm*}
Here $\mathcal{N}_{\hat j}$ and $\mathfrak{N}^\vee$ are certain K-theoretic classes and $\Theta$ denotes their elliptic Thom class, $\tau(-h\delta)$ is a shift of K\"ahler parameters and $\pi^A!$ is an isomorphism. The operator $\Shuffle$ permutes the Chern roots of the varieties $\naka_{Q}(\vi',\w')$ and $\naka_{Q}(\vi'',\w'')$ by shuffles of type $(\vi_i',\vi_i'')$ for every $i\in I$. A more precise statement together with a detailed explanation of the notation will be given in the next sections.

The significance of this result is due to the fact that given a chamber $\mathfrak{C}$ there exist $m-1$ choices of decompositions $\w=\w'+\w''$ and $\vi=\vi'+ \vi''$ that satisfy the requirements of the theorem, where $m=\sum_{i}\w_i$. As a consequence, the stable envelopes of the Nakajima varieties $\naka_{Q}(\vi,\delta_i)$, where $\delta_i\in \N^I$ are the fundamental dimension vectors, defined as the dimension vectors with a one in the $i$-th entry and zeros elsewhere, determine the stable envelopes of arbitrary Nakajima varieties of the same type $Q$.
This result generalizes the inductive construction of stable envelopes of the cotangent bundles of Grassmannias $T^*\Gr{k}{n}$ proposed in \cite{Felder2018}.

Inspired by that paper, we might call the previous formula \emph{shuffle product of stable envelopes}.
The result also suggests an intimate connection between elliptic stables and the cohomological Hall algebras defined by Kontsevich and Soibelman in \cite{kontsevich2011cohomological}. Indeed, the operator
\begin{equation}
    \label{CoHa operator}
    \Shuffle\left\lbrace \frac{\Theta(\mathcal{N}_{\hat j})}{\Theta(\Nbund^\vee)\Theta(h^{-1}\Nbund^\vee)}\cdot\right\rbrace
\end{equation}
can be seen as the elliptic version of the shuffle operator defining the algebra structure of the cohomological Hall algebra of the preprojective path algebra of the quiver $Q$, cf. \cite{kontsevich2011cohomological}\cite{yang2017quiver}\cite{yang2018coHa}. The connection becomes apparent expressing \eqref{CoHa operator} in terms of the Chern roots.

\subsection{Organization of the paper}

After having presented basic notions about the geometry and topology of Nakajima varieties $\naka_{Q}(\vi,\w)$
and the definition of elliptic stable envelopes in Section \ref{Section 2}, in Section \ref{Section 3} we generalize the construction of the abelianizations of Nakajima varieties as in \cite[Section 4]{aganagic2016elliptic} by considering GIT symplectic reductions of $T^*\Repn{Q}{\vi}{\w}$ by algebraic groups isomorphic to $G_{\vi'}\times G_{\vi''}$, which we call refinements of $\naka_{Q}(\vi,\w)$, and we show that the stable envelopes of these varieties exist and allow to recover the stable envelopes of $\naka_{Q}(\vi,\w)$.

In Section \ref{Section 4}, we prove that the Cartesian product of smaller Nakajima varieties $\naka_{Q}(\vi',\w')\times \naka_{Q}(\vi'',\w'')$ is an embedded subvariety of the refinement $T^*\Repn{Q}{\vi}{\w}////G_{\vi'}\times G_{\vi''}$, and in the following section we show that the stable envelopes $\Stab{C}{F}$ of the refinement 
\[
T^*\Repn{Q}{\vi}{\w}////G_{\vi'}\times G_{\vi''}
\]
with $F\subset \naka_{Q}(\vi',\w')\times \naka_{Q}(\vi',\w')$ can be produced from the ones of $\naka_{Q}(\vi',\w')$ and $ \naka_{Q}(\vi'',\w'')$. This is the content of Proposition \ref{ main proposition }.

Finally, in Theorem \ref{shuffle theorem} of Section \ref{Section 6}, we combine the two aforementioned results to write an explicit formula that produces stable envelopes of the Nakajima variety $\naka_{Q}(\vi,\w)$ from the stable envelopes of $\naka_{Q}(\vi',\w')$ and $\naka_{Q}(\vi'',\w'')$.

In the last section we exploit Theorem \ref{shuffle theorem} to produce inductive formulas for the stable envelopes of some famous Nakajima varieties. First, applying our results to the cotangent bundle of Grassmannians $T^*\Gr{k}{n}$, we give a new proof of the shuffle formula in \cite{felder1994conformal}. Secondly, we exhibit a shuffle formula for stable envelopes of the instanton moduli space $\mathcal{M}(v,w)$ equipped with its standard torus action. Technically speaking, this is not the torus action considered in Theorem \ref{shuffle theorem}, but we will show that for some choices of chambers our formula works without changes. This result, combined with the explicit formula for the stable envelopes of the Hilbert schemes of $v$-points in the plane $\Ci^{2}$ \cite{smirnov2018elliptic}, gives the first explicit formula for the elliptic stable envelopes for $\mathcal{M}(v,w)$ known to the author.
\subsection{Acknowledgments}
I want to thank my advisor Giovanni Felder, for having pushed me to learn about this beautiful subject during my master thesis, for his guidance and for his constant attention throughout the development of this project. I also thank Stefano D'Alesio for his suggestions and friendly explanations and Richard Rimanyi, who drew my attention to a mistake in an earlier version.

This work was supported by the National Centre of Competence in
Research SwissMAP--The Mathematics of Physics-- and grant $200021\_196892$ of the Swiss National
Science Foundation.

\newpage


\section{Definitions and notation}
\label{Section 2}
\subsection{Linear representations of a quiver}

Let $(Q,I)$ be a quiver, i.e. an oriented graph, with finite vertex set $I$, where loops and multiple edges are allowed. The quiver data is encoded in the adjacency matrix $Q=\lbrace Q_{ij} \rbrace$,
where
\[
Q_{ij}=\lbrace\text{number of edges from $i$ to $j$}\rbrace.
\]
By abuse of notation, we identify the set of oriented edges with the matrix $Q$.
There are two natural ``head" and ``tail" maps $h,t:Q\to I$ that associate to an oriented edge $e\in Q$ its head and tail vertices $h(e)$ and $t(e)$ respectively.

A representation of a quiver$(Q,I)$ is an assignment of a vector space $V_i$ to every vertex $i\in I$ and a map $x_{t(e),h(e)}\in \Hom(V_{t(e)},V_{h(e)})$ to every edge $e\in Q$.

Given a dimension vector $\vi\in N^I$, we define the space of representations $\text{Rep}_Q(\vi)$ as the set of representations of $Q$ such that $V_i=\Ci^{\vi_i}$, endowed with the structure of a vector space. In other words we have
\begin{equation}
\label{lin repn}
\text{Rep}_Q(\vi)=\bigoplus_{e\in Q}\Hom(V_{(t(e))}, V_{(h(e))}).
\end{equation}
Given a quiver $Q$, we define the framed doubled quiver $(\overline Q, \overline I)$ as the quiver whose vertex set $\overline I$ consists of two copies of $I$, and whose adjacency matrix is
\begin{equation}
\label{framed dubled quiver}
\overline Q=\begin{pmatrix}
Q+Q^T & 1\\
1 & 0
\end{pmatrix}
\end{equation}
Using equations \eqref{lin repn} and \eqref{framed dubled quiver}, one sees that the space of linear representations $\text{Rep}_{\overline Q}{(\vi,\w)}$, associated to the dimension vectors $\vi, \w\in N^I$, corresponding to the vector spaces assigned to the vertices lying inside the two copies of $I$ in $\overline I$, is isomorphic to
\[
T^*\Repn{Q}{\vi}{\w}= T^*\left(\bigoplus_{e\in Q}\Hom(V_{(t(e))}, V_{(h(e))})\oplus \bigoplus_{i\in I}\Hom(V_i,W_i)\right),
\]
i.e. to the cotangent bundle of
\[
\Repn{Q}{\vi}{\w}:=\bigoplus_{e\in Q}\Hom(V_{(t(e))}, V_{(h(e))})\oplus \bigoplus_{i\in I}\Hom(V_i,W_i).
\]
From now on, we identify $\text{Rep}_{\overline Q}{(\vi,\w)}$ with $T^*\Repn{Q}{\vi}{\w}$.  We equip the vector space $T^*\Repn{Q}{\vi}{\w}$ with the structure of a symplectic vector space, with symplectic form induced by the canonical one form of the cotangent bundle $T^*\Repn{Q}{\vi}{\w}$.

\subsection{Group actions}

On $T^*\Repn{Q}{\vi}{\w}$ many actions can be defined. Firstly, the groups
\[
G_{\vi}:=\prod_{i\in I}\GL(\vi_i) \qquad G_{\w}:=\prod_{i\in I}\GL(\w_i)
\]
act on $\Repn{Q}{\vi}{\w}$ by conjugation and hence also on the linear cotangent bundle $T^*\Repn{Q}{\vi}{\w}$. Whenever two edges are connected by more than one arrow, there is a nontrivial\footnote{nontrivial in the sense that it does not factor through the action of $G_{\vi}$} action of the group
\[
G_{\text{par}}=\prod_{i\neq j}GL(Q_{ij}),
\]
whose restriction to 
\[
\bigoplus_{\substack{e\in Q \\ t(e)=i,h(e)=j}}\Hom(V_t(e),V_{h(e)})\cong \Hom(V_t(e),V_{h(e)})\otimes \Ci^{Q_{ij}}
\]
 is given by the standard action of $GL(Q_{ij})$ on $\Ci^{Q_{ij}}$.
 Similarly, there is an action of 
 \[
 G_{\text{sympl}}:=\prod_{i\in I}Sp(2Q_{ii})
 \]
 whose restriction to 
\[
T^*\left(\bigoplus_{\substack{e\in Q \\ t(e)=h(e)=i}}\Hom(V_t(e),V_{h(e)})\right)\cong\Hom(V_i,V_i)\otimes T^*\Ci^{Q_{ii}}
\]
is given by the standard action of $Sp(2Q_{ii})$ on $T^*\Ci^{Q_{ii}}$.

A direct check shows that all these actions preserve the symplectic form of $T^*\Repn{Q}{\vi}{\w}$.
Finally, we will also consider the action of the one dimensional torus $\Ci_h^\times$ rescaling the cotangent directions with weight $h^{-1}$. The subscript, which is standard in the literature, simply refers to the letter labelling the coordinate of $\Ci^\times$.


\subsection{Maximal tori}
Let $T$ be a maximal torus of $G_\w\times G_{\text{par}}\times G_{\text{sympl}}\times \Ci^\times_h$. We write
\[
T=A_\w\times B\times \Ci^\times_h,
\]
where $A_\w$ is a maximal torus in $\G_\w$ and $B$ is a maximal torus in
$G_\text{par}\times G_\text{sympl}$. As a consequence of the previous observations, it follows that the subtorus $A:=A_\w\times B$ preserves the symplectic form, whereas $\Ci_h^\times $ scales it.

\subsection{Definition of Nakajima varieties}

By $\naka_{Q,\theta}(\vi,\w)$ we denote the Nakajima variety associated to a quiver $(Q,I)$, dimension vectors $\w, \vi\in \N^I$ and stability condition $\theta\in\Z^{I}\cong \text{Char}(G_\vi)$ \cite{ginzburg2009lectures}\cite{Nakajimainstantons}\cite{Nakajimaquiver}\cite{Nakajima_2001}.
Namely, $\naka_{Q,\theta}(\vi,\w)$ is realized as the GIT symplectic reduction of the space of representations $T^*\Repn{Q}{\vi}{\w}$ by the Hamiltonian action of $G_{\vi}$:
\[
\naka_{Q,\theta}(\vi,\w):=T^*\Repn{Q}{\vi}{\w}////^\theta G_{\vi}.
\]
Equivalently, $\naka_{Q,\theta}(\vi,\w)$ is the quotient 
\[
\mu^{-1}(0)^{\theta-ss}/G_\vi,
\]
where $\mu:T^*\Repn{Q}{\vi}{\w}\to \g_\vi^*$ is the moment map of the $G_{\vi}$-action and $\mu^{-1}(0)^{\theta-ss}$ is the open set of $\theta$-semistable representations in $\mu^{-1}(0)$.

Since the actions of $\Ci_h^\times$, $G_{\text{par}}$, $G_{\text{sympl}}$ and $G_{\w}$ commute with the action of $G_{\vi}$, they descend to actions on $\naka_{Q,\theta}(\vi,\w)$. In particular, the subgroup $\Ci^\times_{h}$ rescales the symplectic form while all the other actions preserve it. 

In the rest of these notes, we will be interested in the induced action of a subtorus of the maximal torus $T$.

\subsection{Torus-fixed points}
A striking feature of Nakajima varieties is that the fixed components of the action of a subtorus of $A$ are still quiver varieties \cite{maulik2012quantum}. For our interests it will suffice to have an explicit formula for the $A_\w$-fixed components, which follows inductively from the following proposition \cite[Section 2.4.1]{maulik2012quantum}:
\begin{propn}
Let $\Ci^\times$ be a one dimensional subtorus of $A_\w$ acting with weight one on the subspaces $W'_i\subset W_i$, and trivially on their complements $W''_i$. Then
\begin{equation}
    \label{decomposition fixed points}
    \naka_Q(\vi,\w)^{\Ci^\times}\cong\bigsqcup_{\vi'+\vi''=\vi}\naka_Q(\vi',\w')\times \naka_Q(\vi'',\w''),
\end{equation}
where $\w'$ and $\w''$ are the dimension vectors of $W'$ and $W''$ and the isomorphism is induced by the diagonal inclusion $\Repn{Q}{\vi'}{\w'}\times\Repn{Q}{\vi''}{\w''}\hookrightarrow \Repn{Q}{\vi}{\w}$.
\end{propn}
By iteration, one gets a decomposition of $\naka_Q(\vi,\w)^{A_\w}$ as a product of Nakajima varieties of type $Q$.

\begin{eg}
\label{fixed points Grassmannians}
Let $Q=A_1$, $\vi=(v)$, $\w=(w)$ such that $v<w$ and $\theta>0$. The Nakajima variety $\naka_{Q,\theta}(\vi,\w)$ is isomorphic to the cotangent bundle $T^*\Gr{v}{w}$ of the Grassmannian $\Gr{v}{w}$ and the action of $A_\w=(\Ci^\times)^w$ on $T^*\Gr{v}{w}$ seen as a Nakajima variety coincides with its canonical action on $T^*\Gr{v}{w}$. The fixed components
\[
\naka_{A_1}(v^{(1)},1)\times \dots \times \naka_{A_1}(v^{(w)},1)\subset T^*\Gr{v}{w}
\]
are labelled by strings of positive integers $(v^{(1)},\dots,v^{(w)})$ such that $\sum_j v^{(j)}=v$ and $v^{(j)}\in \{0,1\}$ for $j\in \lbrace 1,\dots,w\rbrace$. As a consequence, fixed points in $T^*\Gr{v}{w}$ are labelled by subsets $\lbrace j_1,\dots, j_v\rbrace \subset \lbrace 1,\dots,w\rbrace$ corresponding to those $v^{(j_i)}$ that are nonzero.

It is easy to check that the fixed point labelled by $\lbrace j_1,\dots, j_v\rbrace$ corresponds to the coordinate plane $\chi_{a_{j_1}}\oplus,\dots\oplus\chi_{a_{j_v}}\subset W=\Ci^{w}$ in the zero section $\Gr{v}{w}$ of $T^*\Gr{v}{w}$.
\end{eg}
\begin{eg}
\label{Jordan quiver}
Consider the Jordan quiver, i.e. the quiver with one vertex and one loop. Given two positive integers $v,w\in \N$, the representation space of the framed doubled quiver is the vector space 
\begin{align*}
    T^*\Repn{Q}{\vi}{\w}&=T^*\left(\Hom(V,V)\oplus \Hom(V,W)\right)\\
    &=\Hom(V,V)\oplus \Hom(V,V)\oplus \Hom(V,W)\oplus \Hom(W,V),
\end{align*}
where $V=\Ci^v$ and $W=\Ci^w$, so that a point in $T^*\Repn{Q}{\vi}{\w}$ consists of a tuple $(x,y,i,j)$ that can be represented as follows:
\begin{equation*}
    \begin{tikzcd}
    V\arrow[loop left, distance=3em, start anchor={[yshift=-1ex]west}, end anchor={[yshift=1ex]west}, "x"]\arrow[loop right, distance=3em, start anchor={[yshift=1ex]east}, end anchor={[yshift=-1ex]east}, "y"]\arrow[d, shift left,  "i"] \\
    W\arrow[u, shift left, "j"]
    \end{tikzcd}
\end{equation*}
The moment map is given by 
\[
\mu:T^*\Repn{Q}{\vi}{\w}\to \mathfrak{gl}(v)\cong\mathfrak{gl}(v)^* , \qquad (x,y,i,j)\mapsto [x,y]+ji,
\]
and for a choice of character $\theta\in \Char(GL(V))$, $\theta(g)=\det(g)^\theta$ such that $\theta<0$, we have \cite{Nakajimaquiver}
\[
\naka_{Q}(v,w)=\left\{ (x,y,i,j)\ \middle\vert \begin{array}{l}
    [x,y]+ji=0,\\
    \text{there is no nontrivial subspace $S\subset V$} \\
    \text{stabilized by $x$ and $y$ and such that $im(j)\subset S$}
  \end{array}\right\}/GL(V).
\]
It is well known that this is the ADHM construction of the moduli space $\mathcal{M}(v,w)$ of framed rank $w$ torsion free sheaves with second Chern class equal to $v$, see \cite{Nakajimalectureshilbert}\cite{AIF_1956__6__1_0}. By ``framed" we mean that this is the moduli space of torsion free sheaves $\mathcal{F}$ equipped with an isomorphism 
\[
\restr{\mathcal{F}}{\Pe^1}\cong \mathscr{O}_{\Pe^1}^w,
\]
where $\Pe^1$ is a prescribed linear subvariety of $\Pe^2$. 
In particular, if $w=1$, this variety is the Hilbert scheme of $v$-points in $\Ci^2$, which we denote by $\Hilb_v$.
There are two natural actions on $\mathcal{M}(v,w)$: an action of $\GL(w)=G_\w$ changing the frame and an action of $\GL(2)$ induced by the action of $\Pe^2$ that restricts to the linear action on $\Ci^2=\Pe^2\setminus \Pe^1$.
The actions of $G_\w$ and $\text{SL}(2)\cong \text{Sp}(2)\subset\GL(2) $ coincide with their actions on $\mathcal{M}(v,w)$ seen as a Nakajiama variety, while the action of the one dimensional subtorus of $\GL(2)$ of matrices of the form 
\[
\begin{pmatrix}
1 & 0\\
0 & h
\end{pmatrix}
\]
coincides with the action of $\Ci^\times_h$.
By equation \eqref{decomposition fixed points} we have 
\[
\mathcal{M}(v,w)^{A_\w}=\bigsqcup_{v^{(1)}+\dots+v^{(w)}=v}\Hilb_{v^{(1)}}\times\dots\times\Hilb_{v^{(w)}}
\]
and hence
\begin{equation}
    \label{fixed set innstantons}
    \mathcal{M}(v,w)^{A_\w\times B}=\bigsqcup_{v^{(1)}+\dots+v^{(w)}=v}\Hilb_{v^{(1)}}^B\times\dots\times\Hilb_{v^{(w)}}^B.
\end{equation}

The fixed-points set $\Hilb_v^B$ of the Hilbert scheme $\Hilb_v$ is finite and consists of a union of zero dimensional Nakajima varieties for the quiver $A_\infty$ labelled by partitions $\lambda=(\lambda_1,\dots,\lambda_l)$ of $v$, see \cite{smirnov2018elliptic} or \cite{smirnov2013instanton}. It follows that also the set $\mathcal{M}(v,w)^{A\times B}$ is finite and its points are labelled by tuples $(\lambda^{(1)},\dots,\lambda^{(w)})$ where each  $\lambda^{(j)}$ is a partition of $v^{(j)}$ and $v^{(1)}+\dots+v^{(w)}=v$. Moreover, the tuple $(v^{(1)},\dots,v^{(w)})$ determines the $A_\w$-fixed component
\[
\Hilb_{v^{(1)}}\times\dots\times\Hilb_{v^{(w)}}\subset \mathcal{M}(v,w)
\]
where the fixed points labelled by those tuples $(\lambda^{(1)},\dots,\lambda^{(w)})$ such that $\sum_m\lambda_m^{(j)}=v^{(j)}$ for all $j$ lie. In the following, we identify the $A_\w\times B$-fixed points of $\mathcal{M}(v,w)$ with such tuples of partitions. 

\end{eg}

\subsection{Topology of Nakajima varieties}
A representation $U$ of $G_\vi$ allows to define a vector bundle $\mathscr{U}$ over $X$ as follows
\begin{equation}
\label{tautological bundles naka}
    \mathscr{U}=\mu^{-1}(0)^{\theta-ss}\times U/G_{\vi}.
\end{equation}
In particular, this gives tautological bundles $\mathscr{V}_i$ associated to the vector spaces $V_i$.
Additionally, we define the topologically trivial $A$-equivariant vector bundles
\[
\W_i=\naka_{Q,\theta}(\vi,\w)\times W_i
\]
associated to the framing vector spaces $W_i$.

In this paper we consider in the $T$-equivariant elliptic cohomology of Nakajima varieties associated to an elliptic curve $E=\Ci^\times/q^{\Z}$, cf. \cite{Ganter_2014}\cite{ginzburg1995elliptic} for a comprehensive description, and \cite[Appendix A]{okounkov2020inductiveI} for a review in the context of stable envelopes of Nakajima varieties, whose notation  we closely follow.
We will be only interested in its $0$-th (covariant) functor
\[
E_T(-):T\text{-spaces}\to \text{Schemes}_{E_T},
\]
where
$E_T:=E_T(\pt)$ is the cohomology of the one point space.
For convenience, we will denote the functorial morphism $E_T(f):E_T(X)\to E_T(Y)$ associated to a map $f:X\to Y$ simply by $f$. 
A $T$-equivariant complex vector bundle $V\to X$ of rank $r$ defines a map 
\[
c(V):E_T(X)\to E^{(r)}
\]
called elliptic characteristic class of $V$, see \cite[Section (1.8)]{ginzburg1995elliptic} and \cite[Section 5]{Ganter_2014}. Its coordinates on the target, which is the symmetrized $r$-th cartesian power of $E$, are called elliptic Chern roots.
Associated to $V$ is an invertible sheaf, called Thom sheaf of $V$
\[
\Theta(V):=c^*D_V,
\]
where $D_V$ is the divisor $D_V=\lbrace0\rbrace+E^{(r-1)}\subset E^{(r)}$, formed by those tuples $(t_1,\dots, t_r)\in E^{(r)}$ that contain $0$. Since the divisor $D_V$ is effective, the line bundle $\Theta(V)$ admits a canonical global section, which following the literature we denote in the same way. Since $\Theta(V\oplus V')=\Theta(V)\otimes\Theta(V')$, the assignment $V\mapsto \Theta(V)$ extends to a group homomorphism
\[
K_T(X)\mapsto \Pic{}{E_T(X)}.
\]
If $f$ is a proper morphism of smooth varieties, we can also define its pushforward
\[
f!:f_*\Theta(-N_f)\to \mathcal{O}_{E_T(Y)},
\]
where $\Theta(-N_f)=\Theta(N_f)^{-1}$ is the inverse of the Thom sheaf of the normal bundle of $f$.  

The following elliptic version of Kirwan surjectivity \cite{McGerty_2017}, which in  singular cohomology essentially states that the ring $H^*(\naka_Q(\vi,\w))$ is generated as a $H^*_T(\pt)$-algebra by the Chern roots of the tautological line bundles, was proved in \cite[Section 2.5.2]{aganagic2016elliptic}:
\begin{thm}
The elliptic characteristic classes associated to the tautological bundles \eqref{tautological bundles naka}
\[
c(\mathscr{V}_i):E_T(\naka_{Q}(\vi,\w))\to E^{(\vi_i)}
\]
induce a morphism
\begin{equation}
\label{char emb}
    \chi_{c}:E_T(\naka_Q(\vi,\w))\hookrightarrow E_T\times\prod_{i\in I}E^{(\vi_i)}
\end{equation}
that is an embedding.
\end{thm}

The equivariant Picard group of a Nakajima variety admits a satisfactory description:
\begin{propn}[\cite{Knop_Picard}]
Let $X$ be a Nakajima variety. The following holds:
\begin{enumerate}
    \item The Picard group $\Pic{}{X}$ is a finitely generated free abelian group generated by the determinant line bundles $\det \V_i$ of the tautological bundles $\V_i$ on $X$.
    \item The equivariant Picard group $\Pic{T}{X}$ fits in the following split exact sequence of abelian groups
\[
\begin{tikzcd}
0\arrow[r] & \Char(T)\arrow[r] &  \Pic{T}{X}\arrow[r] & \Pic{}{X}\arrow[r] &0.
\end{tikzcd}
\]
In particular, it is a finitely generated free abelian group.
In the sequence, the first nontrivial map sends a character to the corresponding topologically trivial line bundle and the last one forgets the equivariant structure.
\end{enumerate}

\end{propn}
\begin{remq}
Even though there is a tautological line bundle for every vertex $i\in I$ of the quiver, the previous theorem does not say that the rank of $\Pic{}{X}$ is $|I|$. This is because two determinant line bundles $\V_i$ and $\V_{j}$ might be isomorphic. This always happens when $\dim\naka_Q(\vi,\w)=0$, but it is not the only case. Indeed, consider the quiver $A_2$ with dimension vectors $w=(2,0)$ and $v=(1,1)$ and stability parameter $\theta=(\theta_1,\theta_2)$, with $\theta_1, \ \theta_2>0$.
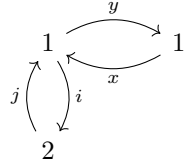
\begin{equation*}
    \begin{tikzcd}[row sep=large, column sep = large]
1\arrow[d,bend left, "i"] \arrow[r,bend left, "y"] & 1   \arrow[l, bend left, "x"] \\ 
2\arrow[u,bend left, "j"]
\end{tikzcd}
\end{equation*}
It is well known that in this case $\naka_{A_2,\theta}(\vi,\w)=T^*\mathbb{P}^1$, and hence $\Pic{}{\naka_{A_2,\theta}(\vi,\w)}=\Z$, even though there are two tautological line bundles $\det\V_1=\V_1$ and $\det\V_2=\V_2$. This is because they are isomorphic. To see this, notice that for every point $[i,j, x,y]\in \naka_{A_2,\theta}(\vi,\w)$ the map $x$ is an isomorphism (because of the stability condition) and hence the morphism of line bundles 
\[
\V_2\ni [i,j,x,y,v]\mapsto [i,j,x,y,x(v)]\in \V_1
\]
is an isomorphism.
\end{remq}
\begin{eg}
\label{picard JOrdan quiver}
The construction of the instanton moduli space $\mathcal{M}(v,w)$ (and in particular of the Hilbert scheme $\Hilb_v=\mathcal{M}(v,1)$) as a quiver variety associated to the Jordan quiver shows that its Picard group $\Pic{}{\mathcal{M}(v,w)}$ is freely generated by the class $\det \V$, where $\V$ is the tautological bundle associated to the unique vertex of the Jordan quiver. Since this bundle is nontrivial, it follows that
\[
\Pic{}{\mathcal{M}(v,w)}\cong \Z.
\]
\end{eg}

\subsection{Chambers and Attracting sets}
Let $X$ be a smooth quasi-projective symplectic variety with an action a torus $T$ rescaling the symplectic form with weight $h^{-1}$ and let $A=\ker(h)$. In this work we will be interested in the case when $X$ is a Nakajima variety or one of the varieties that we will associate to it in the next section.

\begin{defn}[\cite{maulik2012quantum}]
A chamber $\mathfrak{C}$ is a connected component of $\text{Cochar}(A)\otimes_\Z \R\setminus \Delta$, where $\Delta$ is the hyperplane arrangement determined by the $A$-weights of the normal bundles of $X^A$.

We say that a weight $\chi\in \Char(A)$ is attracting (resp. repelling) with respect to $\mathfrak{C}$ if $\lim_{t\to 0} \chi\circ \sigma (t)=0$ (resp. if $\lim_{t\to 0} \chi\circ \sigma (t)=\infty$) for any $\sigma\in \mathfrak{C}$. The trivial character is called the $A$-fixed character.
\end{defn} 

Restricted to $F$, any $A$-equivariant $K$-theoretic class $V$ of $X$ decomposes according to the characters of $A$ in attracting, $A$-fixed and repelling classes
\[
\restr{V}{F}=\restr{V}{F,>}+\restr{V}{F,0}+\restr{V}{F,<}.
\]
For instance, we have that 
\[
\restr{TX}{F}=\restr{TX}{F,>}+TF+\restr{TX}{F,<}.
\]
If $X$ is a Nakajima variety and $A=A_\w$, these geometrically-defined chambers coincide with the Lie-theoretic chambers.
Indeed, the $A_\w$-weights of the normal bundles of the $A_\w$-fixed components of a Nakajima variety $X$ are the roots of the group $\GL(\sum_j \w_j)$, and hence the hyperplane arrangement $\Delta$ coincides with the one that defines the Lie-theoretic chambers of a maximal torus of $\GL(\sum_j \w_j)$.

Let $F$ be a fixed component in $X^{A}$. The attracting set 
\[
\Att{C}{F}=\Set{x\in X}{\lim_{t\to0} \sigma(t)\cdot x\in F, \text{ for some $\sigma \in \mathfrak{C}$}}
\]
is well defined, i.e. independent of the choice of $\sigma\in \mathfrak{C}$, and it is also an affine bundle over $F$. This follows from the standard results in \cite{BBdecomposition}. 

\subsection{Elliptic stable envelopes}

Let $X$ be as above and let $\Epic{T}{X}=\Pic{T}{X}\otimes_{\Z}E$. Elliptic stable envelopes \cite{aganagic2016elliptic} \cite{okounkov2020inductiveI} are morphisms of certain coherent sheaves over the base
\[
\Base{X}{T}:=E_T\times\Epic{T}{X}.
\]
The definition of elliptic stable envelopes requires two choices:
\begin{enumerate}
	\item a polarization $T^{1/2}X$, i.e. a class $T^{1/2}X\in K_T(X)$ such that $ TX=T^{1/2}X+h^{-1}T^{1/2}X^\vee$.
	There is a canonical choice of polarization for a Nakajima variety $X=T^*\Repn{Q}{\vi}{\w}////G_{\vi}$, which is
	\begin{equation}
    \label{standard polarization naka}
    T^{1/2}X=\Repn{Q}{\V}{\W}-\g_{\V},
\end{equation}
	where
	\[
	\g_{\V}=\bigoplus_{i\in I }\Hom(\V_i,\V_i)
	\]
	and
	\[
	\Repn{Q}{\V}{\W}=\bigoplus_{e\in Q}\Hom(\V_{(t(e))}, \V_{(h(e))})\oplus \bigoplus_{i\in I}\Hom(\V_i,\W_i)
	\]
	\item a chamber $\mathfrak{C}\subset\text{Cochar}(A)\otimes_\Z \R$ for the action of the torus $A$.
\end{enumerate}

The polarization $T^{1/2}X$ induces a polarization $T^{1/2}X^{A}$ on the fixed locus $X^{A}$, given by the $A$-fixed component of the restriction $\restr{T^{1/2}X}{X^{A}}$. Moreover, we define the index $\ind$ associated to a polarization $T^{1/2}X$ and a chamber as
\[
\ind:=\restr{T^{1/2}X}{X^A,>}.
\]

Let $s$ be a section of some coherent sheaf over $E_T(X)$, and let $Y\subset X$ be a $T$-invariant subset. We say that $s$ is supported on $Y$ if $f^*s=0$, where
\[
f:E_T(X\setminus Y)\to E_T(X)
\]
is the functorial map associated to the inclusion $X\setminus Y\hookrightarrow X$.

\begin{defn}[\cite{aganagic2016elliptic}]
\label{defn stable envelopes}
    Elliptic stable envelopes are morphisms of sheaves over the base $\Base{T}{X}$
    \[
    \Theta(T^{1/2}X^{A})\otimes \tau(-h\det\ind)^*\mathscr{U}_{X^{A}} \to \Theta(T^{1/2}X)\otimes \mathscr{U}_X\otimes \Theta(h)^{\rk(\ind)},
    \]
    holomorphic in coordinates on $E_A$, meromorphic in coordinates on $\Base{X}{T}/E_A$,
    that satisfy the following two axioms:
    \begin{enumerate}
        \item \label{Condition 1 stable env}If the section $s$ is supported on $F$, then $\text{Stab}_{\mathfrak C}(s)$ is supported on the full attracting set $\text{Att}^f_{\mathfrak{C}}(F)$, which is the minimal closed subset of X that contains $F$ and is closed under taking $\Att{C}{-}$.
        
        \item \label{condition 2 stable env}Let
        \begin{equation*}
            \begin{tikzcd}
            F& \Att{C}{F}\arrow[l, swap, "m"]\arrow[r, "j"]& X\setminus (\text{Att}^f_{\mathfrak{C}}(F)\setminus \Att{C}{F})
            \end{tikzcd}
        \end{equation*}
        denote, respectively, the structure morphism of the affine bundle $m:\Att{C}{F}\to F$ and the inclusion. Let also $i:F\hookrightarrow X$ denote the inclusion of the connected component $F$ in $X$. Then for every section $s$ supported on $F$, we have\footnote{This definition of stable envelopes differs by a sign $(-1)^{\rk(\ind)}$ from the one in \cite{aganagic2016elliptic}. The pushforward $j!$ is well defined because the inclusion $j$ is closed and hence proper.}
        \begin{equation}
            \label{condition 2 stab}
            i^*\circ \text{Stab}(s)=i^*\circ j!\circ m^*(s).
        \end{equation}
    \end{enumerate}
In the definition, $\rk(\ind)$ stands for the rank of the index class $\ind$ and $\Theta(h)$ is the Thom bundle of the topologically trivial line bundle acted on by $\Ci^\times_h$ with weight one.
The universal line bundle $\mathscr{U}_X$ was defined for a general variety in \cite[Section 2.7.1]{aganagic2016elliptic}. For us it will suffice to define it when $X$ is a Nakajima variety and $T=A\times \Ci_h^\times$ is the canonical torus acting on it.
The defining equality $E_T= \Char(T)\otimes_{\Z} E$ induces coordinates $a_l$, $l=1,\dots, \dim(A)$ and $h$ on the cover $\Char(T)\otimes_{\Z}\Ci^\times$ of $E_T$. Similarly, a choice of isomorphism 
\[
\Pic{T}{X}\cong\Char(T)\oplus \Pic{}{X}
\]
induces coordinates  $z_i$, $i\in I$, $z_{a_l}$, $l=1,\dots, \dim(A)$ and $z_h$ on the cover $\Pic{T}{X}\otimes_{\Z} \Ci^\times$ of $\Epic{T}{X}$.
Then the universal line bundle $\mathscr{U}_X$ is the pullback under the characteristic embedding
\[
\chi_c\times id:E_T(X)\times \Epic{T}{X}\hookrightarrow E_T\times\prod_{i\in I}E^{(\vi_i)}\times \Epic{T}{X}
\]
of the line bundle over
\[E_T\times \prod_{i\in I} E^{(\vi_i)}\times \Epic{T}{X}
\]
with the automorphy factor (cf. \cite{oleks2010vector}) of the function
	\begin{equation}
	\label{automorphy universal}
	\prod_{i\in I}\phi(z_i,\det \mathcal{V}_i)\prod_{a=1}^n \phi(z_{a_i}, a_i) \phi(z_h, h),
	\end{equation}
	where \[
	\phi(x,y)=\frac{\theta(xy)}{\theta(x)\theta(y)}.
	\]
	Here $\theta$ is the odd theta function 
	\[
	\theta(x)=(x^{1/2}-x^{-1/2})\prod_{n> 0}(1-q^{n}z)(1-q^{n}x^{-1}).
	\]
\end{defn}
\begin{thm}[\cite{aganagic2016elliptic}, \cite{okounkov2020inductiveI}]
For fixed polarization $T^{1/2}X$ and chamber $\mathfrak{C}$, there exist a unique map that satisfies the conditions of Definition \ref{defn stable envelopes}. 
\end{thm}
This map, denoted by $\text{Stab}_\mathfrak{C}$, is called the elliptic stable envelope of $X$.
Since $X^A$ is in general a disjoint union of connected components
\[
X^A=\bigsqcup_j F_j,
\]
the map
$\text{Stab}_\mathfrak{C}$ is equivalent to a collection of maps $\Stab{C}{F_j}$ for every connected component $F_j$.

We end this section recalling the following explicit description of the translate of the universal bundle, shown in \cite[Section 2.8.3]{aganagic2016elliptic}:
\begin{lma}
\label{automorphy translate universal}
	The line bundle
	\[
	\tau(-h\det\ind)^*\mathscr{U}_{X^A}\otimes \mathscr{U}_{X^A}^{-1}
	\]
	coincides with the pullback under the characteristic embedding  $\chi_c\times id$ of the line bundle with automorphy factor given by 
	\begin{equation}
	\phi(h^{-1},\det\ind).
	\end{equation}
\end{lma}

\subsection{Off-shell stable envelopes}

Since the K-theory of a Nakajima variety (and hence also the one of its $A$-fixed components) is generated by the Chern roots of tautological bundles and the Thom sheaves $\Theta(E)$ for some $E\in K_T(X)$ are pullbacks by the characteristic embedding
\[
\chi_c:E_T(X)\to E_T\times \prod_{i\in I}E^{(\vi_i)}, 
\]
it is possible to define their stable envelopes as morphisms of the pushforward to $\Base{T}{X}$ of those line bundles whose pullbacks by the characteristic embeddings are the line bundles of Definition \ref{defn stable envelopes} $\Theta(T^{1/2}X^{A})\otimes \tau(-h\det\ind)^*\mathscr{U}_{X^{A}}$ and  $\Theta(T^{1/2}X)\otimes \mathscr{U}_X\otimes \Theta(h)^{\rk(\ind)}$.

Since the Picard group $\Pic{}{X}$ is a free abelian subgroup of $\Char(G_\vi)$ and $\Pic{T}{X}\cong \Char(T)\times \Pic{}{X}$, it is also possible to define stable envelopes as morphism of sheaves over the possibly enlarged base $E_T\times (\Char(T)\oplus \Char(G_\vi))\otimes_\Z E$. To recover the original definition, one simply restricts these stable envelopes to $\Base{T}{X}\subset E_T\times (\Char(T)\oplus \Char(G_\vi))\otimes_\Z E$. 
Following the literature, we call this latter version of stable envelopes off-shell stable envelopes. For a thorough description of off-shell stable envelopes see \cite{smirnov2018elliptic}.
\begin{remq}
The off-shell definition allows to obtain nontrivial stable envelopes directly from the data $\vi,\w$ and $Q$ even when the variety $X=\naka_Q(\vi,\w)$ is zero dimensional. This degenerate cases might sound useless at first, because the on-shell restriction of an off-shell stable envelope of a zero dimensional Nakajima variety $X$ is the identity map, but we will see that these off-shell stable envelopes are the building blocks for stable envelopes for nontrivial Nakajima varieties.
\end{remq}
\begin{eg}
\label{off shell zero dim grass}
To give an illustrative example, we consider the following case: let $Q=A_1$, $\vi=\w=\delta_1$, so that $X=\naka_Q(\vi,\w)$ is a singleton, and the torus $A$ is isomorphic to $\Ci^{\times}$.
Since the torus action is trivial, there is no need for a chamber.
The image of the characteristic embedding
\[
E_{T}(X)\hookrightarrow  E_a\times E_h\times E_t
\]
is the subvariety $\lbrace a=t\rbrace $, where $(a,h,t)$ are the coordinates on $E\times E\times E$. 
The Picard group $\Pic{T}{X}$ is simply the character group $\Char(T)$, but the off-shell version is $\Char(T)\oplus \Char(\GL(1))$.
Altogether, we have an embedding
\[
\Base{T}{X}\hookrightarrow (E_a\times E_h)\times (\Char(T)\oplus \Char(\GL(1))\otimes_{\Z} E).
\]
Clearly, there is only one $A$-fixed point, which is $X$ itself, and the class $\ind$ is trivial for every polarization, so there is only one stable envelope:
\[
    \text{Stab}:\mathscr{U}_{X^{A}} \to \Theta(T^{1/2}X)\otimes \mathscr{U}_{X}
\]
This map can be seen as the global section meromorpihc in $z$ of the line bundle over $E_a\times E_h\times E_t\times (\Char(T)\oplus \Char(\GL(1)))\otimes_\Z E$
\[
\Theta(T^{1/2}X)\otimes \mathscr{U}_{X}\otimes \mathscr{U}_{X^A}^{-1}
\]
satisfying $\text{Stab}(a,h,a)=1$. Let $T^{1/2}X$ be the standard polarization. Then the automorphy factors of these line bundles are:
\begin{alignat*}{2}
    &\theta\left(\frac{t}{a}\right)\qquad &&\text{from $\Theta(T^{1/2}X)$}\\
    &\frac{\theta(zt)}{\theta(z)\theta(t)}\frac{\theta(z)\theta(a)}{\theta(za)}\qquad &&\text{from $\mathscr{U}_{X}\otimes \mathscr{U}_{X^A}^{-1}$, cf. equation \eqref{automorphy universal}}.
\end{alignat*}
As a consequence, one can check that $\text{Stab}=\frac{\theta\left(\frac{tz}{a}\right)}{\theta(z)}$ has the correct quasi-periods in $t$ and $a$.
Notice that the on-shell restriction of $\text{Stab}$ is $
\restr{\frac{\theta\left(\frac{tz}{a}\right)}{\theta(z)}}{\{a=t\}}=1
$, as required.
\end{eg}


\section{Refinements of Nakajima varieties}
\label{Section 3}

\subsection{Definitions and first results}

Let $S$ be a maximal torus of $G_\vi$, let $\vi',\vi''\in \N^I$ be dimension vectors such that $\vi'+\vi''=\vi$ and let $G_{\vi',\vi''}$ be an algebraic subgroup of $G_\vi$ containing $S$ and isomorphic to $G_{\vi'}\times G_{\vi''}$.
Let also $i_{G_{\vi',\vi''}}$ denote the inclusion of $G_{\vi',\vi''}$ in $G_{\vi}$. 

The inclusion $i_{G_{\vi', \vi''}}$ induce an action of the group $G_{\vi',\vi''}$ on $T^*\Repn{Q}{\vi}{\w}$. 
The following lemma follows directly from the definitions:
\begin{lma}
\label{moment map refinement}
The action of $G_{\vi',\vi''}$ is Hamiltonian, and the composition
\[
(di_{G_{\vi', \vi''}})^*\circ \mu:T^*\Repn{Q}{\vi}{\w}\to \text{Lie}(G_{\vi'}\times G_{\vi''})^* 
\]
is its moment map.
\end{lma}
Let $\mu_{\vi',\vi''}$ denote the composition $(di_{G_{\vi', \vi''}})^*\circ \mu$. The character $\theta\in \Char({G_{\vi}})$ restricts to  characters of $G_{\vi',\vi''}$ and $S$, which we denote in the same way.
With this convention, we can define the following GIT quotient:
\[
\naka_{Q,\theta}^{G_{\vi',\vi''}}(\vi,\w):=T^*\Repn{Q}{\vi}{\w}////^{\theta}G_{\vi',\vi''}=\mu_{\vi',\vi''}^{-1}(0)//^{\theta} G_{\vi',\vi''}.
\]
\begin{defn}
    We call the symplectic variety $\naka^{G_{\vi',\vi''}}_{Q,\theta}(\vi,\w)$ the $G_{\vi',\vi''}$-refinement of the Nakajima variety $\naka_{Q,\theta}(\vi,\w)$.
\end{defn}

In the following, we will always assume that the character $\theta$ is generic for the actions of  $G_\vi$, $G_{\vi',\vi''}$ and $S$. This ensures that the corresponding semistable loci are actually stable. In particular, this implies that $\naka_{Q,\theta}^{G_{\vi'\vi''}}(\vi,\w)$ is a smooth symplectic variety of dimension
\[
2\dim{\Repn{Q}{\vi}{\w}}-2 \dim{\text{Lie}(G_{\vi'})}-2 \dim{\text{Lie}(G_{\vi''})}.
\]

\begin{propn}
\label{v-twist of embedding }
Let  $G_{\vi',\vi''}$, and $\widetilde G_{\vi',\vi''}$ be two algebraic subgroups of $G_\vi$ isomorphic to $G_{\vi'}\times G_{\vi''}$. Then the varieties $\naka_{Q,\theta}^{ G_{\vi',\vi''}}(\vi,\w)$ and $\naka_{Q,\theta}^{\widetilde G_{\vi',\vi''}}(\vi,\w)$ are isomoprhic.
\end{propn}
\begin{proof}
By definition, both $G_{\vi',\vi''}$ and $\widetilde G_{\vi',\vi''}$ are isomorphic to $
G_{\vi'}\times G_{\vi''}$ so they are conjugated in $G_\vi$, say by $g$.
The action of $g$ on $T^*\Repn{Q}{\vi}{\w}$ gives a linear automorphism along the conjugation
\[
c_g:G_{\vi',\vi''}\to \widetilde G_{\vi',\vi''}\qquad h\mapsto ghg^{-1}
\]
and hence an isomoprhism of the GIT quotients $\naka_{Q,\theta}^{G_{\vi',\vi''}}(\vi,\w)$ and $\naka_{Q,\theta}^{\widetilde G_{\vi',\vi''}}(\vi,\w)$.
\end{proof}

We end this section with some examples:
\begin{eg}
Let the quiver $(Q,I)$ be the Dynkin diagram $A_1$, $\vi=(v)$, $\w=(w)$ such that $v< w$ and $\theta>0$. The Nakajima variety $\naka_{Q,\theta}(\vi,\w)$ is isomorphic to the cotangent bundle of the Grassmannian $T^*\Gr{v}{w}$. Then a $G_{\vi',\vi''}$-refinement of the Nakajima variety $\naka_{Q,\theta}(\vi,\w)$ is isomorphic to $T^*\Gr{v'}{w}\times T^*\Gr{v''}{w}$, where $\vi'=(v')$ and $\vi''=(v'')$.
\end{eg}
The latter was a rather special case, because in general the $G_{\vi',\vi''}$-refinement of the Nakajima variety $\naka_{Q,\theta}(\vi,\w)$ is not isomorphic to $\naka_{Q,\theta}(\vi',\w)\times \naka_{Q,\theta}(\vi'',\w)$, as the next example shows.
\begin{eg}
Let now the quiver $(Q,I)$ be the oriented Dynkin diagram $A_k$, and let 
\begin{align*}
    &\vi=(v_1,\dots,v_k)\\
    &\w=(w,0,\dots,0)
\end{align*}
with $w> v_1> v_2> \dots> v_k$, and $\theta=(1,\dots,1)$. Then the framed doubled quiver looks like
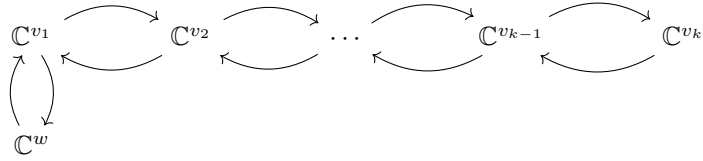
\[
\begin{tikzcd}[row sep=large, column sep = large]
\Ci^{v_1}\arrow[d,bend left] \arrow[r,bend left] & \Ci^{v_2} \arrow[r,bend left]  \arrow[l, bend left] & \cdots\arrow[r, bend left]\arrow[l, bend left] & \Ci^{v_{k-1}}\arrow[l, bend left]\arrow[r,bend left]  & \Ci^{v_k}\arrow[l,bend left]\\ 
\Ci^w\arrow[u,bend left]
\end{tikzcd}
\]
The resulting Nakajima variety $\naka_Q(\vi,\w)$ is isomorphic to the cotangent bundle of the flag variety $\Fl(v_k,\dots,v_1,w)$. In this case, the $G_{\vi',\vi''}$-refinement of the Nakajima variety $\naka_{Q,\theta}(\vi,\w)$ associated to the dimension vectors 
\begin{align*}
    &\vi'=(v_1',\dots,v_k')\\
     &\vi''=(v_1'',\dots,v_k'')
\end{align*}
such that $\vi'+\vi''=\vi$ is not isomorphic to $T^*\Fl(v_k',\dots,v_1',w)\times T^*\Fl(v_k'',\dots,v_1'',w)$ in general. This can be seen by noticing that 
\[
\dim(\naka_{Q,\theta}^{\G_{\vi',\vi''}}(\vi,\w))-\dim(T^*\Fl(v_k',\dots,v_1',n)\times T^*\Fl(v_k'',\dots,v_1'',w))=2\sum_{i=2}^{k-1}v_i' v_{i+1}''+v_i'' v_{i+1}'>0
\]
for a generic choice of vectors $\vi'$ and $\vi''$.
\end{eg}

\subsection{Notation}

We introduce the following more compact notations for Nakajima varieties and their refinements: Given a quiver $Q$, we denote the Nakajima variety $\naka_{Q,\theta}(\vi,\w)$ by $X^\vi_\w$, and its $G_{\vi',\vi''}$-refinement $\naka_{Q,\theta}^{G_{\vi',\vi''}}(\vi,\w)$ by $X^{\vi',\vi''}_\w$.

\subsection{Tautological bundles}
\label{Tautological bundles}
Like Nakajima varieties, also the refinements $X_\w^{\vi',\vi''}$ admit tautological bundles, defined in analogy with equation \eqref{tautological bundles naka}:
\[
\aV_{i}={\mu_{\vi',\vi''}^{-1}(0)^{\theta-ss}\times V_i} /\G_{\vi',\vi''}\qquad \aW_i=X_\w^{\vi',\vi''}\times W_i.
\]
Notice that, since we are taking quotients by the subgroup $G_{\vi',\vi''}\cong G_{\vi'}\times G_{\vi''}$, the vector bundles $\aV_i$ split as direct sums
\begin{equation}
\label{splitting vector bundles}
    \aV_i=\aV'_i\oplus \aV_i'',
\end{equation}
of the vector bundles
\[
\aV_{i}'={\mu_{\vi',\vi''}^{-1}(0)}^{   \theta-ss}\times V'_i /\G_{\vi',\vi''}\qquad \aV''_i={\mu_{\vi',\vi''}^{-1}(0)}^{\theta-ss}\times V_i''/\G_{\vi',\vi''},
\]
where $V_i'$ and $V_i''$ are the subspaces of $V_i$ such that $G_{\vi',\vi''}=\prod_{i\in I}\GL(V_i')\times \GL(V_i'')$.
\subsection{Polarizations}
Like in the case of Nakajima varieties, we can also define the standard polarization for the refinements $X^{\vi',\vi''}_\w=T^*\Repn{Q}{\vi}{\w}////G_{\vi',\vi''}$ as the class
\begin{equation}
    \label{standard polarization refinements}
    T^{1/2}X^{\vi',\vi''}_\w=\Repn{Q}{\aV}{\aW}-\g_{\aV',\aV''},
\end{equation}
where
	\[
	\g_{\aV',\aV''}=\bigoplus_{i\in I }\Hom(\aV'_i,\aV'_i)\oplus\Hom(\aV''_i,\aV''_i)
	\]
	and
	\[
	\Repn{Q}{\aV}{\aW}=\bigoplus_{e\in Q}\Hom(\aV_{(t(e))}, \aV_{(h(e))})\oplus \bigoplus_{i\in I}\Hom(\aV_i,\aW_i)
	\]
The polarization satisfies the equation $TX^{\vi',\vi''}_\w=T^{1/2}X^{\vi',\vi''}_\w+h^{-1}(T^{1/2}X^{\vi',\vi''}_\w)^\vee$ in $K_T(X_\w^{\vi',\vi''})$.

\subsection{Connecting maps}

In the following section, we fix dimension vectors $\vi, \vi',\vi'', \w\in N^I$ such that $\vi'+\vi''=\vi$ and a subgroup $G_{\vi',\vi''}$.
We now build maps relating a Nakajima variety $X^\vi_\w=\naka_Q(\vi,\w)$ with its $G_{\vi',\vi''}$-refinement $X^{\vi',\vi''}_\w=\naka^{G_{\vi',\vi''}}_Q(\vi,\w)$. We will then use the corresponding maps in cohomology to relate the elliptic stable envelopes of $X^\vi_\w$ with the ones of its refinement.

The following result is a straightforward modification of \cite[Section 4.3.1]{aganagic2016elliptic}, which is in turn based on \cite{Shenfeld}.
\begin{propn}
\label{maps relating naka and abel naka}
Let $P$ be a parabolic subgroup containing $G_{\vi',\vi''}\subset G_{\vi}$ and let $\n_{\vi',\vi''}\subset \g_{\vi}$ be the nilpotent subalgebra such that $\mathfrak{p}=\text{Lie}(P)=\g_{\vi'}\oplus\g_{\vi''}\oplus \n_{\vi',\vi''}$.
There exist smooth $T$-varieties $\Fl$ and $\widetilde \Fl$ and smooth morphisms of $T$-varieties $j_+:\Fl\to \widetilde\Fl$, $j_-:\widetilde\Fl\to X^{\vi',\vi''}_\w$ and $\pi:\Fl\to X_\w^{\vi}$ such that 
\begin{enumerate}
    \item the maps $j_+$ and $j_-$ are closed embeddings,
    \item \label{class normal j-}the class of the normal bundle to $j_-$ is $(j_-)^*(h^{-1}\Nbund^\vee)$ where 
    $\Nbund$ is the vector bundle over $X_\w^{\vi',\vi''}$ associated to the nilponent subalgebra $\n_{\vi',\vi''}$:
    \[
    \Nbund=\mu_{\vi',\vi''}^{-1}(0)^{\theta-ss}\times \n_{\vi',\vi''}/G_{\vi',\vi''},
    \]
    \item the map $\pi:\Fl\to X_\w^{\vi}$ is a fiber bundle with fiber $G_\vi/P$,
    \item the class of the bundle $\ker d\pi$ is $(j_-\circ j_+)^*\Nbund^\vee$.
\end{enumerate}
\end{propn}
Explicitly, we have $\Fl=\mu^{-1}(0)//{G_{\vi',\vi''}}$, $\widetilde\Fl=\mu^{-1}(\mathfrak{p}^\perp)//{G_{\vi',\vi''}}$, and the maps $j_+$ and $j_-$ are induced by the closed inclusions
\[
\mu^{-1}(0)\subset \mu^{-1}(\mathfrak{p}^\perp)\subset \mu_{\vi',\vi''}^{-1}(0),
\]
where $\mu:T^*\Repn{Q}{\vi}{\w}\to \g_\vi^*$ is the moment map associated to the action of $G_\vi$.
As it is shown in the diagram below, the maps in Theorem \ref{maps relating naka and abel naka} realize a chain of morphisms relating the Nakajima variety $\naka_Q(\vi,\w)$ to its $G_{\vi',\vi''}$-refinement:
\begin{equation}
\label{L-shaped diagram}
    \begin{tikzcd}
    \Fl\arrow[d, "\pi"]\arrow[r, "j_+"] & \widetilde\Fl\arrow[r, "j_-"] &X_\w^{\vi',\vi''} \\
    X_\w^{\vi}
    \end{tikzcd}
\end{equation}
We remark that the variety $\widetilde\Fl$ and all the maps in the diagram above depend on the choice of the parabolic subgroup $P$.

\subsection{Refinements of stable envelopes}

Being a diagram of morphism of $T$-varieties, The L-shaped diagram \eqref{L-shaped diagram} restricts to the $A$-fixed locus:
\[
\begin{tikzcd}
    \Fl^A\arrow[d, "\pi^{A}"]\arrow[r, "j_{+}^{A}"] & {\widetilde\Fl}^A\arrow[r, "j_{-}^{A}"] &(X^{\vi',\vi''}_\w)^{A} \\
    (X^\vi_\w)^{A}
    \end{tikzcd}
\]

\begin{lma}
\label{lemma choice componets above}
Let $F\subset X$ be an $A$-fixed component and $\mathfrak{C}$ a chamber. For a given choice of parabolic subgroup $P$, there exists a unique $A$-fixed component $F^{\vi',\vi''}$ in $(X_\w^{\vi',\vi''})^{A}$ such that
\begin{enumerate}
    \item $\pi(F^{\vi',\vi''}\cap \Fl)\subset F$
    \item $\restr{\Nbund}{F^{\vi',\vi''}}$ has no repulsive directions, i.e. $\restr{\Nbund}{F^{\vi',\vi''},<}=0$
\end{enumerate}
\end{lma}
\noindent
The first condition in particular implies that the previous diagram restricts further as follows:
\[
\begin{tikzcd}
    F^{\vi',\vi''}\cap\Fl\arrow[d, "\pi^{A}"]\arrow[r, "j_+^{A}"] & F^{\vi',\vi''}\cap{\widetilde\Fl}\arrow[r, "j_-^{A}"] & F^{\vi',\vi''}\\
    F
    \end{tikzcd}
\]
\begin{proof}
Without loss of generality, we assume that the subgroup $G_{\vi',\vi''}\subset G_{\vi}=\prod_{i\in I}\GL(\vi_i)$ consists of block diagonal matrices
\[
\begin{pmatrix}
\GL({\vi_i'}) & 0\\
0 & \GL({\vi_i''})
\end{pmatrix}
\]
and that the parabolic subgroup $P\subset G_{\vi}$ is of the form 
\[
\begin{pmatrix}
\GL({\vi_i'}) & \ast\\
0 & \GL({\vi_i''})
\end{pmatrix}
\]
Since the $G_\vi$-action on the prequotient $\mu^{-1}(0)^{\theta-ss}$ is free, for every fixed point $F\subset X$ there is a ``compensating" map 
\[
\psi_{F}:A\to G_{\vi}
\]
such that the premiage of the fixed component $F$ in $\mu^{-1}(0)$ is fixed by the action of $(a,\psi_{F}(a))\in A\times G_{\vi}$. Also notice that the compensating map is not unique, because conjugation of $\psi_{F}$ by some $g\in G_{\vi}$ gives another compensating map for the same component $F$.

Similarly, there is an analogous map
\[
\psi_{F^{\vi',\vi''}}:A\to G_{\vi',\vi''}
\]
for every fixed component $F^{\vi',\vi''}\subset X_\w^{\vi',\vi''}$.
As a consequence, the first condition of the statement of the lemma is equivalent to the requirement that the composition
\[
\begin{tikzcd}
A\arrow[r, "\psi_{F^{\vi',\vi''}}"]&G_{\vi',\vi''}\arrow[r,hookrightarrow]&G_{\vi}
\end{tikzcd}
\]
is conjugated to $\psi_{F}$, while the
second condition is equivalent to the statement that the action of $A$ on $\mathfrak{n}_{\vi',\vi''}\subset \g_{\vi',\vi''}$ induced by $\psi_{F^{\vi',\vi''}}$ has no repelling weights.
Therefore, it is enough to find a compensating map $\psi_{F^{\vi',\vi''}}$ satisfying the two previous conditions.
The map $\psi_F$ gives a weight decomposition of each vector space $V_i=V_i'\oplus V_i''$:
\[
V_i'=\chi_{i,1}\oplus\dots\oplus \chi_{i,\vi_i'}, \qquad V_i''=\chi_{i,\vi_i'+1}\oplus\dots\oplus \chi_{i,\vi_i}
\]
and the characters $\chi_{i,j}\chi_{i,k}^{-1}$ with $1\leq j\leq \vi_i'$ and $\vi_i''+1\leq k\leq \vi_i$ are the ones that appear in the $A$-action on $\mathfrak{n}_{\vi',\vi''}$ induced by $\psi_F$. 

For every $i\in I $ we choose a permutation $\tau^{(i)}\in \mathfrak{S}_{\vi_i}$ such that $\chi_{i,\tau^{(i)}(j)}\chi_{i,\tau^{(i)}(k)}^{-1}>0$ for $j< k$. Let $\tau=\prod_{i\in I}\tau^{(i)}$.  By construction, the composite map $c_\tau\circ \psi_F$ induces an $A$-action on $\mathfrak{n}_{\vi',\vi''}$ with no repelling weights and defines a fixed point whose image by $\pi$ is exactly $F$. Thus, the map $c_\tau\circ \psi_F$ is the sought-after map $\psi_{F^{\vi',\vi''}}$.
Uniqueness follows because even though compositions of $c_\tau\circ \psi_F$ by permutations $\tau'\in \prod_{i}\mathfrak{S}_{\vi_i'}\times\prod_{i}\mathfrak{S}_{\vi_i''}\subset \prod_{i}\mathfrak{S}_{\vi_i}$ give different maps that induce an $A$ action on $\mathfrak{n}_{\vi',\vi''}$ with no repelling weights, they define the same fixed component in $X_\w^{\vi',\vi''}$.
\end{proof}
The next theorem, which is a straightforward modification of \cite[Section 4.3]{aganagic2016elliptic}, allows to express the stable envelopes $\Stab{C}{F}$ of $X^\vi_\w$ in terms of the stable envelopes $\Stab{C}{F^{\vi',\vi''}}$ of $X^{\vi',\vi''}_{\w}$.
\begin{thm}
\label{theorem abelianization}
The elliptic stable envelope $\Stab{C}{F}$ with standard polarization $T^{1/2}X^\vi_\w$ fits in the following diagram:
\begin{equation}
\label{stable envelopes }
    \begin{tikzcd}[column sep=huge]
    \Theta(T^{1/2}F)\otimes \tau^*\mathscr{U}_{(X^\vi_\w)^{A}}\arrow[rr, "j^{A}_-!\circ ((j_+^{A})^*)^{-1} \circ(\pi_{A}!)^{-1}"]\arrow[d, "\Stab{C}{F}"]& & \Theta(T^{1/2}F^{\vi',\vi''})\otimes \tau^*\mathscr{U}_{(X_\w^{\vi',\vi''})^{A}} \arrow[d, "\lambda^*\Stab{C}{F^{\vi',\vi''}}"]\\
    \Theta(T^{1/2}X^\vi_\w)\otimes \mathscr{U}_{X^\vi_\w}\otimes \Theta(h)^{\rk(\ind)}& & \Theta(T^{1/2}X_\w^{\vi',\vi''})\otimes \mathscr{U}_{X_\w^{\vi',\vi''}}\otimes \Theta(h)^{\rk(\ind')}\arrow[ll, "\pi!\circ j_+^*\circ (j_-!)^{-1}"]
    \end{tikzcd}
\end{equation}
with the following notation:
\begin{itemize}
    \item $T^{1/2}X_\w^{\vi',\vi''}$ is the polarization $T^{1/2}X_\w^{\vi',\vi''}=T^{1/2}X^{\vi',\vi'',\text{standard}}_\w+h^{-1}\Nbund^\vee-\Nbund$,
    \item $\ind'$ is the index class of $T^{1/2}X_\w^{\vi',\vi''}$,
    \item $\tau^*\mathscr{U}_{(X^\vi_\w)^{A}}$ and  $\tau^*\mathscr{U}_{(X_\w^{\vi',\vi''})^{A}}$ stand for $\tau(-h\det\ind )^*\mathscr{U}_{(X^\vi_\w)^{A}}$ and $\tau(-h \det\ind')^*\mathscr{U}_{(X_\w^{\vi',\vi''})^{A}}$.
\end{itemize}
\end{thm}
Before defining the map $\lambda$ in the label of the right vertical arrow in the diagram, notice that, whereas $\Stab{C}{F}$ is a morphism of sheaves over $\Base{T}{X^{\vi}_\w}=E_T\times (\Pic{T}{X^\vi_\w}\otimes_\Z E)$, $\Stab{C}{F^{\vi',\vi''}}$ is a morphism of sheaves over $\Base{T}{X_\w^{\vi',\vi''}}=E_T\times (\Pic{T}{X_\w^{\vi',\vi''}}\otimes_\Z E)$. This apparent inconsistency is avoided by pulling back $\Stab{C}{F^{\vi',\vi''}}$ by the morphism $\lambda:\Base{T}{X^\vi_\w}\to\Base{T}{X_\w^{\vi',\vi''}}$ induced by the group homomorphism
\begin{equation*}
\begin{tikzcd}
\Pic{T}{X_\w^\vi}\subseteq\text{Char}(G_\vi)\oplus\text{Char}(T)\arrow[r, "i^*\oplus id"]& \text{Char}(G_{\vi',\vi''})\oplus\text{Char}(T)\supseteq \Pic{T}{X_\w^{\vi',\vi''}},
\end{tikzcd}
\end{equation*}where $i^*$ is the functorial map determined by the inclusion $i:G_{\vi',\vi''}\hookrightarrow G_\vi$. 
\begin{remq}
Both the stable envelopes of $X_{\w}^{\vi}$ and $X_{\w}^{\vi',\vi''}$ can be obtained from the stable envelopes of their common abelianization \cite[Section 4.2]{aganagic2016elliptic}
\[
X_S:=\Repn{Q}{\vi}{\w}////^\theta S=\mu_S^{-1}(0)//^\theta S,
\]
i.e. of the symplectic reduction of $\Repn{Q}{\vi}{\w}$ with respect to the action of a maximal torus $S\subset G_{\vi',\vi''}\subset G_{\vi}$, whose moment map is denoted by $\mu_S$. 

In both cases, namely when $X=X_{\w}^{\vi}$ or $X=X_{\w}^{\vi',\vi''}$, the construction is essentially the same as before: firstly, one chooses a Borel subgroup $B$ in $G_{\vi',\vi''}$ (or in $G_{\vi}$ if $X=X_{\w}^{\vi',\vi''}$) containing $S$, with Lie algebra $\Lie(B)=\Lie(S)\oplus\n$ and constructs a diagram of $T$-varieties
\begin{equation}
\label{L-shaped diagram S}
    \begin{tikzcd}
    \Fl_S\arrow[d, "\pi"]\arrow[r, "j_+"] & \widetilde\Fl_S\arrow[r, "j_-"] &X_S \\
    X
    \end{tikzcd}
\end{equation}
where the maps in the diagram satisfy properties analogous to the ones of Theorem \ref{maps relating naka and abel naka}; then, having fixed a component $F$ in $X$ and a chamber $\mathfrak{C}$, we have the analogous of Lemma \ref{lemma choice componets above}:
\begin{lma}
\label{lemma choice componets above 2}
For given choices of Borel subgroup $B$, chamber $\mathfrak{C}$ and fixed component $F\subset X^{A}$, there exists a unique ${A}$-fixed component $F_S$ in $(X_S)^{A}$ such that
\begin{enumerate}
    \item $\pi(F_S\cap \Fl_S)\subset F$
    \item $\restr{\Nbund}{F_S}=\restr{\left(\mu_S^{-1}(0)^{\theta-ss}\times \mathfrak{n}/S\right)}{F_S}$ has no repulsive directions, i.e. $\restr{\Nbund}{{F_S},<}=0$. 
\end{enumerate}
\end{lma}
\noindent
Finally, one checks that the diagram 
\begin{equation}
\label{diagram abelianization pseudoabelianization}
    \begin{tikzcd}[column sep=huge]
    \Theta(T^{1/2}F)\otimes \tau^*\mathscr{U}_{(X)^{A}}\arrow[rr, "j^{A}_-!\circ ((j_+^{A})^*)^{-1} \circ(\pi^{A}!)^{-1}"]\arrow[d, "\Stab{C}{F}"]& & \Theta(T^{1/2}F_S)\otimes \tau^*\mathscr{U}_{(X_S)^{A}} \arrow[d, "\lambda^*\Stab{C}{F_{S}}"]\\
    \Theta(T^{1/2}X)\otimes \mathscr{U}_{X}\otimes\Theta(h)^{\rk(\ind)} & & \Theta(T^{1/2}X_S)\otimes \mathscr{U}_{X_{S}}\otimes\Theta(h)^{\rk(\ind_S)}\arrow[ll, "\pi!\circ j_+^*\circ (j_-!)^{-1}"]
    \end{tikzcd}
\end{equation}
commutes. 
In this case
\begin{itemize}
    \item $T^{1/2}X_S$ is the polarization $T^{1/2}X_S^{\text{standard}} +h^{-1}\Nbund^\vee-\Nbund$, where
    \[
    \Nbund=\mu_S^{-1}(0)^{\theta-ss}\times \mathfrak{n}/S,
    \]
    \item $\ind_S$ is the index class of $T^{1/2}X_S$,
    \item $\tau^*\mathscr{U}_{(X)^{A}}$ and  $\tau^*\mathscr{U}_{(X_S)^{A}}$ are shortcuts for $\tau(-h\det\ind )^*\mathscr{U}_{(X)^{A}}$ and $\tau(-h \det\ind_S)^*\mathscr{U}_{(X_S)^{A}}$,
    \item the map $\lambda:\Base{T}{X}\to\Base{T}{X_S}$ is induced by the group homomorphism
    \begin{equation*}
    \begin{tikzcd}
    \Pic{T}{X}\subseteq\text{Char}(G)\oplus\text{Char}(T)\arrow[r, "i^*\oplus id"]& \text{Char}(S)\oplus\text{Char}(T)\supseteq \Pic{T}{X_S},
    \end{tikzcd}
    \end{equation*}
    where $G=G_\vi$ if $X=X^{\vi}_\w$, $G=G_{\vi',\vi''}$ if $X=X^{\vi',\vi''}_\w$ and $i^*$ is the functorial map determined by the inclusion $i:S\hookrightarrow G$.
\end{itemize}
 Notice that this argument can be also used to show the existence stable envelopes of $X_\w^{\vi',\vi''}$ in complete analogy to the proof of the existence of stable envelopes for a Nakajima variety given in \cite[Section 4.3]{aganagic2016elliptic}.
\end{remq}


\section{Embeddings of Nakajima varieties}

\label{Section 4}
\label{subsection embedding nakajima varieties}

The choice of a subgroup $G_{\vi',\vi''}\subset G_{\vi}$ isomoprhic to $G_{\vi'}\times G_{\vi''}$ induces a splitting of the $I$-graded vector space $V$ into the a direct sum $V=V'\oplus V''$. Choose also a splitting of the $I$-graded $A_\w$-module $W$ in graded submodules $W'$ and $W''$ such that the Nakajima varieties $\naka_Q(\vi',\w')$ and $\naka_Q(\vi'',\w'')$ are non-empty.
Associated to these splittings, consider the diagonal inclusion
\begin{equation}
    \label{diagonal inclusion preimages}
    T^*\Repn{Q}{\vi'}{\w'}\times T^*\Repn{Q}{\vi''}{\w''}\to T^*\Repn{Q}{\vi}{\w}
\end{equation}
sending a pair of representations $(x',x'')$ to their direct sum $(x'\oplus x'')$. More explicitly, the restriction of  this map to $\Hom(V_i',V_j')\times \Hom(V_i'',V_j'')\subset T^*\Repn{Q}{\vi'}{\w'}\times T^*\Repn{Q}{\vi''}{\w''}$ sends a pair of linear maps $(x'_{ij},x''_{ij})$ to the block-diagonal map 
\[
x'_{ij}\oplus x''_{ij}=\begin{pmatrix}
x'_{ij} & 0\\
0 & x''_{ij}
\end{pmatrix}\in \Hom(V_i'\oplus V_i'',V_j'\oplus V_j''),
\]
and is defined similarly on $\Hom(V'_i,W'_i)\times \Hom(V''_i, W''_i)$ and $\Hom(W'_i, V'_i)\times \Hom(W''_i, V''_i)$ for all $i\in I$. 

The next lemma follows straightforwardly from the definitions:
\begin{lma}
\label{commuting square inclusion naka}
The following diagram commutes:
\begin{equation*}
    \begin{tikzcd}
T^*\Repn{Q}{\vi'}{\w'}\times T^*\Repn{Q}{\vi''}{\w''} \arrow[d, "\mu_{\vi'}\times \mu_{\vi''}"]\arrow[r] &T^*\Repn{Q}{\vi}{\w}\arrow[d, "\mu_\vi"]\\
\g^*_{\vi'}\times \g^*_{\vi''}& \g^*_\vi \arrow[l, "(di)^*"]
\end{tikzcd}
\end{equation*}
\end{lma}
Commutativity of the previous diagram gives a closed $G_{\vi',\vi''}$-equivariant closed embedding 
\[
\mu^{-1}_{\vi'}(0)\times \mu^{-1}_{\vi''}(0)\hookrightarrow \mu^{-1}_{\vi',\vi''}(0).
\]
\begin{lma}
Let $\theta\in\Char( G_{\vi',\vi''})$, and let $\theta'$ and $\theta''$ be its restrictions to $\Char(G_{\vi'})$ and $\Char(G_{\vi''})$ respectively. Then the previous $G_{\vi',\vi''}$-equivariant embedding restricts to an embedding of the semistable locus:
\begin{equation}
\label{diagonal embedding inclusion}
\mu^{-1}_{\vi'}(0)^{\theta'-ss}\times \mu^{-1}_{\vi''}(0)^{\theta''-ss}\hookrightarrow \mu^{-1}_{\vi',\vi''}(0)^{\theta-ss}
\end{equation}
\end{lma}
\begin{proof}
Given $(p',p'')\in \mu^{-1}_{\vi'}(0)^{\theta'-ss}\times \mu^{-1}_{\vi''}(0)^{\theta''-ss}$, we show that $p'\oplus p'' \in \mu^{-1}_{\vi',\vi''}(0)$ is semistable exploiting King's characterization of semistability \cite{king}. Assume by contradiction that $p'\oplus p''$ is not semistable. This means that there exist a cocharacter $\lambda\in \text{Cochar}(G_{\vi'}\times G_{\vi''})$ such that $\lim_{t\to 0^+}\lambda(t)(p'\oplus p'',z)=(p_0'\oplus p_0'',0)$ for every $z\neq 0$. The cocharacter $\lambda$ is the same as a pair of cocharacters $(\lambda',\lambda'')\in\text{Cochar}(G_{\vi'})\times\text{Cochar}(G_{\vi''})$. Under this identification, we have \[\lambda(t)(p'\oplus p'',z)=(\lambda'(t)p'\oplus \lambda''(t)p'', \chi(\lambda'(t))\chi(\lambda''(t))z).\] The existence of the limit implies in particular that at least one monomial between $\chi(\lambda'(t))$ and $\chi(\lambda''(t))$ is a positive power of $t$. Without loss of generality, we can assume that it is $\chi(\lambda'(t))$. This means that $\lim_{t\to 0^+} \lambda'(t)(p',z)=(p'_0,0)$, i.e. that $p'$ is not semistable, which contradicts the assumption. 
\end{proof}
Taking the quotient by $G_{\vi',\vi''}$, the inclusion $\mu^{-1}_{\vi'}(0)^{\theta'-ss}\times \mu^{-1}_{\vi''}(0)^{\theta''-ss}\hookrightarrow \mu^{-1}_{\vi',\vi''}(0)^{\theta-ss}$ gives a smooth embedding
\begin{equation}
\label{embedding product naka}
    j:X_{\w'}^{\vi'}\times X_{\w''}^{\vi''}\hookrightarrow X_\w^{\vi',\vi''}.
\end{equation}

\begin{eg}
\label{grassmannian case}
Consider the Nakajima variety $T^*\Gr{v}{w}$ with its usual $A=(\Ci^\times)^{w}$-action. In this case, every $G_{\vi',\vi''}$-refinement is of the form $T^*\Gr{v'}{w}\times T^*\Gr{v''}{w}$ with $v'+v''=v$ and the embeddings $X_{\w'}^{\vi'}\times X_{\w''}^{\vi''}\hookrightarrow X_\w^{\vi',\vi''}$ are the embeddings
\[
T^*\Gr{v'}{w'}\times T^*\Gr{v''}{w''}\hookrightarrow T^*\Gr{v'}{w}\times T^*\Gr{v''}{w}
\]
associated to choices of $w'$ and $w''$-dimensional coordinate  subspaces of $W=\Ci^w$.
\end{eg}

\begin{remq}
\label{Emebdded nakajima varieties as torus-fixed component}
The subvariety $X_{\w'}^{\vi'}\times X_{\w''}^{\vi''}$ can be also seen as a fixed component of the action of the one dimensional subgroup $\Ci^\times\subset A_\w$ acting on $W_i'$ with weight one and trivially on $W_i''$ for all $i\in I$. Indeed, since the action of $G_{\vi',\vi''}$ on the prequotient $\mu_{\vi',\vi''}^{-1}(0)^{\theta-ss}$ is free, associated to each fixed component is a compensating map
\[
\psi:\Ci^\times \to G_{\vi',\vi''}
\]
such that the preimage of a fixed component in $X_\w^{\vi',\vi''}$ is fixed by the action of the points $(a,\psi(a))\in \Ci^\times \times G_{\vi}$. It is easy to see that $\psi$ induces a weight decomposition of $V_i=\widetilde V_i'\oplus \widetilde V_i''$ such that $\Ci^\times$ acts on $\widetilde V_i'$ with weight one and trivially on $\widetilde V_i''$. Moreover, the preimages of fixed points are exactly those representations of $Q$ that are in the image of the diagonal inclusion \eqref{diagonal inclusion preimages},
determined by the splittings $V_i=\widetilde V_i'\oplus \widetilde V_i''$ and $W=W_i'\oplus W_i''$. Therefore, the subvariety $X_{\w'}^{\vi'}\times X_{\w''}^{\vi''}\subset X_\w^{\vi',\vi''}$ is the connected component corresponding to the splitting $V_i'=\widetilde V_i'$ and $V_i''=\widetilde V_i''$ for all $i\in I$. Notice that, since there are many possible different splittings of $V$ induced by the action of $\Ci^\times$, in general $\Ci^\times$-fixed connected components of $X_\w^{\vi',\vi''}$ are not products of Nakajima varieties.
\end{remq}
\begin{remq}
\label{remark equivariance j}
The previous remark also shows that the group $T=A_\w \times B \times \Ci_h^\times$ acts on $X_{\w'}^{\vi'}\times X_{\w''}^{\vi''}\subset X_\w^{\vi',\vi''}$ and there is a one dimensional torus fixing it. Moreover, this $T$-action on $X_{\w'}^{\vi'}\times X_{\w''}^{\vi''}$ coincides with the action of $T =T'\times_{(B \times\Ci^\times)}T''$ induced on the Cartesian product $X_\w^{\vi'}\times X_\w^{\vi''}$ by the actions of the subgroups $T'=A_{\w'}\times B \times\Ci^\times$ and $T''= A_{\w''}\times B \times\Ci^\times$ on  the Nakajima varieties $X^{\vi'}_{\w'}$ and $X_{\w''}^{\vi''}$ respectively. As a consequence, we can define an action of $T$ on $X_{\w'}^{\vi'}\times X_{\w''}^{\vi''}$ that makes the embedding $j$ equivariant.  From now on, we will always look at $j$ as a $T$-equivariant embedding.
\end{remq}

We end this section by listing a few properties of the normal bundle of the inclusion $j$ and of the attracting set of $X^{\vi'}_{\w'}\times X^{\vi'}_{\w'}\subset X^{\vi',\vi''}_{\w}$. 

\begin{lma}
\label{lemma normal bundles}
The following holds:
\begin{enumerate}
\item The class of the normal bundle $\mathscr{N}_j$ in $K_T(X_\w^{\vi'}\times X^{\vi''}_{\w''})$ is isomorphic to \[
\mathscr{M}(\V',\V'',\W',\W'')+h^{-1}\mathscr{M}(\V',\V'',\W',\W'')^\vee\in K_T(X_{\w'}^{\vi'}\times X_{\w''}^{\vi''}),\] where $\mathscr{M}(\V',\V'',\W',\W'')$ is the class
\begin{equation}
\label{class M}
\bigoplus_{i\in I}\left(\Hom(\V_i', \W_i'')\oplus h^{-1}\Hom(\W_i', \V_i'')\right)\bigoplus_{e\in Q}\left( \Hom(\V_{t(e)}', \V_{h(e)}'')\oplus h^{-1}\Hom(\V_{h(e)}', \V_{t(e)}'')\right).
\end{equation}
\item Let $ \overline{\mathfrak{C}}$ be a chamber of the torus $\Ci^\times $ fixing $X^{\vi'}_{\w'}\times X^{\vi''}_{\w''}\subset X^{\vi',\vi''}_{\w}$ as in Remark \ref{Emebdded nakajima varieties as torus-fixed component}. The normal bundle $\mathscr{N}_{j}$ splits as the direct sum of vector bundles
\[
\mathscr{N}_{j}=\mathscr{N}_{j}^>\oplus \mathscr{N}_{j}^<
\]
such that $\mathscr{N}_{j}^>$ is the bundle whose fibers are attracting with respect to $\overline{\mathfrak{C}}$ and $\mathscr{N}_{j}^<=h^{-1}(\mathscr{N}_{j}^>)^\vee$.
\item The normal bundle $\mathscr{N}_{\hat j}$ of the attracting manifold $\Att{\overline{C}}{X_{\w'}^{\vi'}\times X_{\w''}^{\vi''}}$
\[\begin{tikzcd}
X_{\w'}^{\vi'}\times X_{\w''}^{\vi''} \arrow[d, hookrightarrow] \arrow[r,"j", hookrightarrow] & X_\w^{\vi',\vi''}\\
\Att{\mathfrak{\overline C}}{X_{\w'}^{\vi'}\times X_{\w''}^{\vi''}}\arrow[ur, swap, "\hat j", hookrightarrow]&
\end{tikzcd}
\]
satisfies
\[
\restr{\mathscr{N}_{\hat j}}{X_{\w'}^{\vi'}\times X_{\w'}^{\vi''}}=\mathscr{N}_{j}^<.
\]
\item \label{induced polarization}Let $T^{1/2}X_{\w'}^{\vi'}$ and $T^{1/2}X_{\w''}^{\vi''}$ be standard polarizations of  $X_{\w'}^{\vi'}$ and $X^{\vi''}_{\w''}$, defined in \eqref{standard polarization naka}. Replacing the tautological bundles $\V_i'$, $\V_i''$, $\W_i'$ and $\W_i''$ on $X_{\w'}^{\vi'}$ and $X_{\w''}^{\vi''}$ in their definition with the tautological bundles $\aV_i'$, $\aV_i''$, $\aW_i'$ and $\aW_i''$ on $X^{\vi',\vi''}_\w$ defined in Section \ref{Tautological bundles}, we produce classes in $K_T(X^{\vi',\vi''}_\w)$, which we denote in the same way. With this abuse of notation, the classes
\begin{align*}
    & T^{1/2}X_{\w'}^{\vi'}+T^{1/2}X_{\w''}^{\vi''}+ \mathscr{M}(\aV',\aV'',\aW',\aW'')\\
    & T^{1/2}X_{\w'}^{\vi'}+T^{1/2}X_{\w''}^{\vi''}+ h^{-1}\mathscr{M}(\aV',\aV'',\aW',\aW'')^\vee
\end{align*}
are both polarizations of $X_\w^{\vi',\vi''}$.
\end{enumerate}
\end{lma}
\begin{proof}
Recall the definitions of the tautological bundles in Section \ref{Tautological bundles}.
The first claim follows from the explicit formulas for the classes $TX^{\vi',\vi''}_\w$, $TX^{\vi'}_{\w'}$ and $TX^{\vi'}_{\w'}$ coming from \eqref{standard polarization naka} and \eqref{standard polarization refinements}, from the short exact sequence of the normal bundle $\mathcal{N}_j$ and the fact that
\begin{align*}
    &j^*\aV=j^*(\aV'\oplus \aV'')=\V'\oplus\V''\\
    &j^*\aW=j^*(\aW'\oplus \aW'')=\W'\oplus\W''.
\end{align*}

The second and third points follow straightforwardly from the fact that $X^{\vi'}_{\w'}\times X^{\vi''}_{\w''}$ is a $\Ci^\times$-fixed component and from the definition of $\Att{\overline{ 
\mathfrak{C}}}{X_{\w'}^{\vi'}\times X_{\w''}^{\vi''}}$. Finally, the fourth is an easy computation.
\end{proof}
\begin{remq}
\label{remark extension bundles}
Notice that the vector bundles $\mathscr{M}(\V',\V'',\W',\W'')$ and  $h^{-1} \mathscr{M}(\V',\V'',\W',\W'')^\vee$ extend to bundles on the whole $X^{\vi',\vi''}_\w$ by formula \eqref{class M}.
Moreover, $\mathcal{N}_{\hat j}$ is isomorphic to the restriction of one of these two vector bundles, depending on the choice of the chamber $\overline{\mathfrak{C}}$. Whenever such a choice of chamber is fixed, we will identify $\mathcal{N}_{\hat j}$ with the corresponding vector bundle on $X^{\vi',\vi''}_\w$. This will help to highlight the geometric meaning of the formulas. 
\end{remq}


\section{Main technical result}
\label{Section 5}

The goal of this section is to produce, under certain conditions, stable envelopes of the $T$-variety  $X_{\w}^{\vi',\vi''}$ from the ones of $X_{\w'}^{\vi'}$ and $X_{\w''}^{\vi''}$.
The basic idea is to use a chain of maps to relate the stable envelopes of $X_{\w}^{\vi',\vi''}$ to the ones of the subvariety $X_{\w'}^{\vi'}\times X_{\w''}^{\vi''}$, which can be readily obtained from the ones of the $T'$-variety $X_{\w'}^{\vi'}$ and the $T''$-variety $X_{\w''}^{\vi''}$.
\subsection{Tensor product of stable envelopes}
Recall that stable envelopes of a $T$-variety $X$ are morphisms of certain sheaves over the base $\Base{X}{T}=E_T\times \Epic{T}{X}$. As a consequence, to relate stable envelopes of $X_{\w}^{\vi',\vi''}$, $X_{\w'}^{\vi'}$, and $X_{\w''}^{\vi''}$ we need a morphism
\begin{equation}
\label{morphism needed}
    \Base{X_{\w}^{\vi',\vi''}}{T}\to \Base{X_{\w'}^{\vi'}}{T'}\times \Base{X_{\w''}^{\vi''}}{T''}
\end{equation}
We now build this morphism by defining maps
\begin{align*}
    &\phi:E_T\to E_{T'}\times E_{T''}\\
    &\psi:\Epic{T}{X_{\w}^{\vi',\vi''}}\to \Epic{T}{X_{\w'}^{\vi'}}\times \Epic{T}{X_{\w''}^{\vi''}}.
\end{align*}
The morphism $\phi$ is simply defined functorially from the quotients $T\to T/A_{\w''}=T'$ and $T\to T/A_{\w'}=T''$.
As for $\psi$, we begin with the following lemma:
\begin{lma}
The homomorphism
\[
\begin{tikzcd}
\Pic{}{X_{\w}^{\vi',\vi''}}\arrow[r, "j^*"]& \Pic{}{X_{\w'}^{\vi'}\times X_{\w''}^{\vi''}}\arrow[rr, "(pr_1^*\times pr_2^*)^{-1}"] && \Pic{}{X_{\w'}^{\vi'}}\times\Pic{}{X_{\w''}^{\vi''}}
\end{tikzcd}
\]
is surjective and sends a generator $\det(\aV'_i\oplus \aV''_j)$ to a pair of generators $(\det\V_i, \det\V_j)$.
\end{lma}
\begin{proof}
Both claims readily follows from commutativity of the following diagram
\[
\begin{tikzcd}
\Pic{}{X_{\w}^{\vi',\vi''}}\arrow[r, "j^*"]& \Pic{}{X_{\w'}^{\vi'}\times X_{\w''}^{\vi''}}\arrow[rr, "(pr_1^*\times pr_2^*)^{-1}"] && \Pic{}{X_{\w'}^{\vi'}}\times\Pic{}{X_{\w''}^{\vi''}}\\
\Char(G_{\vi',\vi''})\arrow[u, two heads] \arrow[rrr, "(pr_1^*\times pr_2^*)^{-1}"] & &&\Char(G_{\vi'})\times \Char(G_{\vi''}),\arrow[u, two heads] 
\end{tikzcd}
\]
where the lower horizontal map $(pr_1^*\times pr_2^*)$ is the functorial isomorphism associated to projections
\begin{align*}
&pr_{1}:G_{\vi',\vi''}=\prod_{i\in I}{\GL(V_i')\times\GL(V_i'')}\to \prod_{i\in I}{\GL(V_i')}=G_{\vi'}\\  &pr_{2}:G_{\vi',\vi''}=\prod_{i\in I}{\GL(V_i')\times\GL(V_i'')}\to \prod_{i\in I}{\GL(V_i'')}=G_{\vi''}.
\end{align*}
\end{proof}
\noindent
From the exact sequences 
\[
\begin{tikzcd}
0\arrow[r]&\Char(T) \arrow[r]&\Pic{T}{X^\vi_\w}\arrow[r]& \Pic{}{X_\w^\vi}\arrow[r]&0\\
0\arrow[r]&\Char(T) \arrow[r]&\Pic{T}{X_\w^{\vi',\vi''}}\arrow[r]& \Pic{}{X_\w^{\vi',\vi''}}\arrow[r]&0\\
\end{tikzcd}
\]
we build the following diagram with exact rows:
\begin{equation}
\label{map picards exterior tenso}
  \begin{tikzcd}[column sep=small]
0\arrow[r]& \Char(T)\arrow[r]\arrow[d] &\Pic{T}{X_\w^{\vi',\vi''}}\arrow[r]\arrow[d]& \Pic{}{X_\w^{\vi',\vi''}}\arrow[r]\arrow[d, "(pr_1^*\otimes pr_2^*)^{-1}\circ j^*"]&0 \\
0\arrow[r]& \Char(T')\oplus\Char(T'')\arrow[r] &\Pic{T}{X_{\w'}^{\vi'}}\oplus \Pic{T}{X_{\w''}^{\vi''}}\arrow[r]& \Pic{}{X_{\w'}^{\vi'}}\oplus\Pic{}{X_{\w''}^{\vi''}}\arrow[r]&0.
\end{tikzcd}  
\end{equation}
Here the left vertical map is induced by the inclusions $T'\hookrightarrow T$ and $T''\hookrightarrow T$. Since the outer maps are surjective, it follows from the five lemma that the central map is an surjective too.
Tensoring this map with $E$, we get the sought-after morphism of abelian varieties
\[
\psi:\Epic{T}{X_\w^{\vi',\vi''}}\to \Epic{T}{X_{\w'}^{\vi'}}\times \Epic{T}{X_{\w''}^{\vi''}}.
\]

\begin{defn}
\label{defn tensor product nakajima varieties}
Let $\Stab{C'}{F^{\vi'}}$ and $\Stab{C''}{F^{\vi''}}$ be stable envelopes of $X_{\w'}^{\vi'}$ and $X_{\w''}^{\vi''}$ respectively.
Their exterior product is the following morphism of sheaves over $\Base{X_\w^{\vi',\vi''}}{T}$: \begin{equation}
\label{tensor product stab}
    \Stab{C'}{F^{\vi'}}\boxtimes\Stab{C''}{F^{\vi''}}:=(\phi\times \psi)^*(pr_1^*\Stab{C'}{F^{\vi'}}\otimes pr_2^*\Stab{C''}{F^{\vi''}}).
\end{equation}

Notice that the sheaf theoretic exterior product $pr_1^*\Stab{C'}{F^{\vi'}}\otimes pr_2^*\Stab{C''}{F^{\vi''}}$ defines a morphism of sheaves over $\Base{X_{\w'}^{\vi'}}{T'}\times \Base{X_{\w''}^{\vi''}}{T''}$, that we pull-back to define a morphism over $\Base{X_\w^{\vi',\vi''}}{T}$.
\end{defn}

\subsection{Lifts}
The exterior product \eqref{tensor product stab} is the fundamental building block to produce stable envelopes of $X_{\w}^{\vi',\vi''}$ for those fixed components lying in the image of $j$ starting from the ones of $X_{\w'}^{\vi'}$ and $X_{\w''}^{\vi''}$. However, there is no chance for this formula to give the stable envelope for $X_{\w}^{\vi', \vi''}$. This is because the fibers of the normal bundle of a fixed component in $j$ generally consist both of normal directions in $X_{\w'}^{\vi'}\times X_{\w''}^{\vi''}\subset X^{\vi',\vi''}_\w$ and in its complement. 
Let $\Ci^\times $ be the torus fixing $X^{\vi'}_{\w'}\times X^{\vi''}_{\w''}\subset X^{\vi',\vi''}_{\w}$ as in Remark \ref{Emebdded nakajima varieties as torus-fixed component}, and let $\overline {\mathfrak{C}}$ be the projection of $\mathfrak{C}$ on $\text{Cochar}(\Ci^\times)\otimes_\Z\R$. The natural solution is to exploit the triangle lemma in \cite[Section 3.6]{aganagic2016elliptic} to construct the stable envelopes of $X_{\w}^{\vi',\vi''}$ as the composition 
\[
\Stab{C}{F^{\vi'}\times F^{\vi''}}=\text{Stab}_{\overline{\mathfrak{C}}}(X_{\w'}^{\vi'}\times X_{\w''}^{\vi''})\circ \Stab{C'}{F^{\vi'}}\boxtimes\Stab{C''}{F^{\vi''}}.
\] 
The purpose of the remaining of this section is to express the first map of this composition more explicitly.

By construction
\begin{equation}
\label{tensor stab compact notation}
\Stab{C'}{F^{\vi'}}\boxtimes\Stab{C''}{F^{\vi''}}: \mathcal{L}_A\to \mathcal{L}
\end{equation}
is a morphism of $\mathcal{O}_{\Base{X^{\vi',\vi''}_{\w}}{T}}$-modules, where $\mathcal{L}_A$ is a line bundle over $E_T(F^{\vi'}\times F^{\vi''})\times \Epic{T}{X^{\vi',\vi''}_{\w}}$ and $\mathcal{L}$ is a line bundle over $E_T(X^{\vi'}_{\w'}\times X^{\vi''}_{\w''})\times\Epic{T}{X^{\vi',\vi''}_{\w}}$ uniquely determined by the defining properties of stable envelopes. 
\begin{lma}
\label{pull back line bundle}
The line bundle $\mathcal{L}$ is pulled back from $E_T(X_{\w}^{\vi',\vi''})\times \Epic{T}{X^{\vi',\vi''}_{\w}}$ via the map $j\times id: E_T(X_{\w'}^{\vi'}\times X_{\w''}^{\vi''})\times \Epic{T}{X^{\vi',\vi''}_{\w}}\to E_T(X_{\w''}^{\vi',\vi''})\times \Epic{T}{X^{\vi',\vi''}_{\w}}$ induced by the inclusion $j:X_{\w''}^{\vi''}\times X_{\w''}^{\vi''}\hookrightarrow X_{\w}^{\vi',\vi''}$.
\end{lma}

\begin{proof}
Consider the following commutative diagram
\begin{equation*}
\label{diagram lifts bundles}
\begin{tikzcd}
E_T(X_{\w'}^{\vi'}\times X_{\w''}^{\vi''})\times \Epic{T}{X_{\w'}^{\vi',\vi''}}\arrow[d, "j\times id"]\arrow[r, "\chi^c\times id" ]& E_T\times \prod_{i\in I} E^{(\vi')}\times E^{(\vi'')}\times \Epic{T}{X_{\w'}^{\vi',\vi''}}\\
E_T(X_{\w''}^{\vi',\vi''})\times \Epic{T}{X^{\vi',\vi''}_{\w}} \arrow[ur] &
\end{tikzcd}
\end{equation*}
where the top horizontal map is induced by the characteristic embedding \eqref{char emb} of the product of Nakajima varieties $X_{\w'}^{\vi'}\times X_{\w''}^{\vi''}$ and the diagonal map is the product of the characteristic classes of the tautological bundles on $X^{\vi',\vi''}_{\w}$. Commutativity follows from functoriality of characteristic classes.
By definition of stable envelopes, the target of $\Stab{C'}{F^{\vi'}}\boxtimes\Stab{C''}{F^{\vi''}}$ is a line bundle pulled back by the horizontal map of the commutative diagram and hence it factors through $E_T(X_{\w}^{\vi',\vi''})\times \Epic{T}{X_{\w'}^{\vi',\vi''}}$.
\end{proof}
Let $\widehat{\mathcal{L}}$ be the line bundle on $E_T(X_{\w''}^{\vi',\vi''})\times \Epic{T}{X^{\vi',\vi''}_{\w}}$ that pulls back to $\mathcal{L}$ according to Lemma \ref{pull back line bundle} and
let us denote by
\[
(\Stab{C'}{F^{\vi'}}\boxtimes\Stab{C''}{F^{\vi''}})^{lift}:\mathcal{L}_A\to \widehat{\mathcal{L}}
\]
a morphism of sheaves over $\Base{T}{X^{\vi',\vi'}_{\w}}$ that, when restricted to $\mathcal{L}$, coincides with \eqref{tensor stab compact notation}. We call it a lift of $\Stab{C'}{F^{\vi'}}\boxtimes\Stab{C''}{F^{\vi''}}$.

\begin{propn}
\label{ main proposition }
Let $F^{\vi'}$ and $F^{\vi''}$ be fixed components in $X_{\w'}^{\vi'}$ and $X_{\w''}^{\vi''}$ respectively, and let $\mathfrak{C}'$ and $\mathfrak{C}''$ be the chambers in $\text{Cochar}(A')\otimes_\Z\R$ and $\text{Cochar}(A'')\otimes_\Z\R$ induced by the inclusions $A'\hookrightarrow A$ and $A''\hookrightarrow A$ respectively.
Assume also the following:
\begin{assumption}
\label{Assumption main theorem}
The restriction of the normal bundle $\mathscr{N}_{\hat j}$ of the attracting manifold 
\[
\hat j:\Att{\overline{C}}{X_{\w'}^{\vi'}\times X_{\w''}^{\vi''}}\hookrightarrow X^{\vi',\vi''}_\w
\]
fits in the following exact sequence of bundles over $F^{\vi'}\times F^{\vi''}$:
\begin{equation*}
    \begin{tikzcd}
    0\arrow[r]& \restr{\mathcal{N}^{X_{\w'}^{\vi'}\times X_{\w''}^{\vi''}}_{\Att{C}{F^{\vi'}\times F^{\vi''}}}}{F^{\vi'}\times F^{\vi''}}\arrow[r]& \restr{\mathcal{N}^{X_{\w}^{\vi',\vi''}}_{\Att{C}{F^{\vi'}\times F^{\vi''}}}}{F^{\vi'}\times F^{\vi''}}\arrow[r]& \restr{\mathcal{N}_{\hat j}}{F^{\vi'}\times F^{\vi''}} \arrow[r] & 0
    \end{tikzcd},
\end{equation*}
where $\mathcal{N}^{X_{\w'}^{\vi'}\times X_{\w''}^{\vi''}}_{\Att{C}{F^{\vi'}\times F^{\vi''}}}$ is the normal bundle of $\Att{C}{F^{\vi'}\times F^{\vi''}}$ in $X_{\w'}^{\vi'}\times X_{\w''}^{\vi''}$ and $\mathcal{N}^{X_{\w}^{\vi',\vi''}}_{\Att{C}{F^{\vi'}\times F^{\vi''}}}$ is the normal bundle of $\Att{C}{F^{\vi'}\times F^{\vi''}}$ in $X_{\w}^{\vi',\vi''}$.
\end{assumption}
\noindent
Then for any choice of lift $\Stab{C'}{F^{\vi'}}\boxtimes\Stab{C''}{F^{\vi''}})^{lift}$ we have
\begin{equation}
\label{formula stab refinement}
    \Stab{C}{F^{\vi'}\times F^{\vi''}}=\Theta(\mathcal{N}_{\hat j})(\Stab{C'}{F^{\vi'}}\boxtimes \Stab{C''}{F^{\vi''}})^{lift},
\end{equation}
where $\Stab{C}{F^{\vi'}\times F^{\vi''}}$ is the stable envelope of the fixed component $F^{\vi'}\times F^{\vi''}\subset X_\w^{\vi',\vi''}$ with polarization
\[
T^{1/2}X_{\w'}^{\vi'}+T^{1/2}X_{\w''}^{\vi''}+\mathcal{N}_{\hat j},
\]
defined in Lemma \ref{lemma normal bundles} and Remark \ref{remark extension bundles}.
\end{propn}
We defer the proof of this proposition to the appendix. Here we collect some remarks.

\begin{remq}
\label{remark existence lifts}
It is always possible to construct, at least locally over $\Base{T}{X^{\vi',\vi''}_\w}/E_A$, a lift 
\[
(\Stab{C'}{F^{\vi'}}\boxtimes\Stab{C''}{F^{\vi''}})^{lift}
\] and hence to apply equation \eqref{formula stab refinement} locally. The resulting maps satisfy the defining properties of the stable envelope $\Stab{C}{F^{\vi'}\times F^{\vi''}}$ and hence glue by uniqueness. Existence of local lifts follows from the abelianization procedure \cite[Section 4.3]{aganagic2016elliptic} for the stable envelopes $\Stab{C'}{F^{\vi'}}$ and $\Stab{C''}{F^{\vi''}}$, which allows to produce locally over $\Base{T}{X^{\vi',\vi''}_\w}/E_A$ the latter maps off-shell. This strategy will be explained in more detail in the appendix.
\end{remq}

\begin{remq}
\label{assuption strong but good}
Assumption \ref{Assumption main theorem} is indeed a strong one. It implies that the $\mathfrak{C}$-repelling directions of $F^{\vi'}\times F^{\vi''}\subset  X_{\w'}^{\vi'}\times X_{\w''}^{\vi''}\subset X_\w^{\vi',\vi''}$ that do not lie in $X_{\w'}^{\vi'}\times X_{\w''}^{\vi''}$ are exactly the $\overline{\mathfrak{C}}$-repelling directions for $X_{\w'}^{\vi'}\times X_{\w''}^{\vi''}$. In other words, it means that 
\[
\Att{C}{F^{\vi'}\times F^{\vi''}}\subset \Att{\overline{C}}{X_{\w'}^{\vi'}\times X_{\w''}^{\vi''}}.
\]
Luckily, in the next section we will show that this condition is adequate to produce stable envelopes of Nakajima varieties $X^{\vi}_{\w}$ from the ones of $X_{\w'}^{\vi'}$ and $X_{\w''}^{\vi''}$ when $A=A_\w$.
\end{remq}

\begin{eg}
It is easy to apply formula \eqref{formula stab refinement} whenever the stable envelopes $\Stab{C''}{F^{\vi''}}$ and $\Stab{C''}{F^{\vi''}}$ are expressed off-shell. Let us consider the easiest example here. More involved examples will be provided in the next section. Let $Q=A_1$, $\vi'=(1)$, $\vi''=(0)$, $\w'=\w''=(1)$ and $\mathfrak{C}=\lbrace a_1> a_2\rbrace$. Then  $X^{\vi',\vi''}_\w=T^*\mathbb{P}^1$, the fixed point $F^{\vi'}\times F^{\vi''}$ is the north pole and by Example \ref{off shell zero dim grass} the stable envelopes of $F^{\vi'}=X^{\vi'}_{\w'}$ and $F^{\vi''}=X^{\vi''}_{\w''}$ are 
\[
\Stab{C'}{F^{\vi'}}(t,a_1,z',h)=\frac{\theta\left(\frac{tz'}{a_1}\right)}{\theta(z')}\qquad \Stab{C''}{F^{\vi''}}(t,a_2,z'',h)=1.
\]
Hence the stable envelope $\Stab{C}{F^{\vi'}\times F^{\vi''}}$ with polarization as in Proposition \ref{ main proposition } is
\[
\Stab{C}{F^{\vi'}\times F^{\vi''}}(t,a_1,a_2,z',z'',h)=\Theta(\mathcal{N}_{\hat j})\frac{\theta\left(\frac{tz'}{a_1}\right)}{\theta(z')}=\theta\left(\frac{a_2}{t}\right)\frac{\theta\left(\frac{tz'}{a_1}\right)}{\theta(z')}.
\]
\end{eg}


\section{Shuffle product of Stable envelopes}
\label{Section 6}

\subsection{Preliminary observations}

Theorem \ref{theorem abelianization} allows to build stable envelopes of a Nakajima variety $X_{\w}^{\vi}$ from the ones of $X_\w^{\vi',\vi''}$. On the other hand, Proposition \ref{ main proposition } allows to build some stable envelopes of $X_{\w}^{\vi',\vi''}$ from the ones of the smaller Nakajima varieties $X_{\w'}^{\vi'}$ and $X_{\w''}^{\vi''}$. In this section we combine these results to produce stable envelopes of $X_{\w}^{\vi}$ from the ones of $X_{\w'}^{\vi'}$ and $X_{\w''}^{\vi''}$. Before entering in the mathematical details, we give a brief overview of the problems to tackle.

Fix a component $F\subset (X_\w^{\vi})^{A}$. Recall that Theorem \ref{theorem abelianization} allows to build $\Stab{C}{F}$ from the stable envelope $\Stab{C}{F^{\vi',\vi''}}$ for a specific choice of fixed component $F^{\vi',\vi''}\subset X_{\w}^{\vi',\vi''}$ uniquely determined by a choice of a parabolic subgroup $P\subset G_{\vi}$ containing $G_{\vi',\vi''}\subset G_{\vi}$, as stated by Lemma \ref{lemma choice componets above}, while Proposition \ref{ main proposition } allows to build stable envelopes of $X_\w^{\vi',\vi''}$ for those components in the image of an embedding $X^{\vi'}_{\w'}\times X_{\w''}^{\vi''}\hookrightarrow X_{\w}^{\vi',\vi''}$ that satisfies Assumption \ref{Assumption main theorem}. As a consequence, to combine these results, we need have to solve the following
\begin{problem}
Find decompositions $V=V'\oplus V''$ and $W=W'\oplus W''$ leading to an embedding $j:X^{\vi'}_{\w'}\times X^{\vi''}_{\w''}\hookrightarrow X^{\vi',\vi''}_{\w}$ and a parabolic group $P\supset G_{\vi',\vi''}$ such that
\begin{enumerate}
    \item the component $F^{\vi',\vi''}\subset X_{\w}^{\vi',\vi''}$ associated to $F$ by Lemma \ref{lemma choice componets above} is in the image of $j:X^{\vi'}_{\w'}\times X^{\vi''}_{\w''}\hookrightarrow X^{\vi',\vi''}_{\w}$;
    \item these data satisfy Assumption \ref{Assumption main theorem} of Proposition \ref{ main proposition }.
\end{enumerate}
\end{problem}

To satisfy these conditions simultaneously, we need to restrict to the special but important case of $A=A_\w$. As a consequence, from now on we will always assume that the torus $A$ acting on $X^\vi_\w$ is $A_\w$. This does not mean that the quiver defining $X^\vi_\w$ cannot have loops or multiple edges, but just that we ignore the action of the torus $B$ arising from these kinds of edges. 
Nonetheless, in Section \ref{subsection Stable envelopes of the instanton moduli space} we will show that in some cases we can still combine these two conditions when the $B$-action is taken into account.

\subsection{Solution to the problem}
In the next lemma, we show that there are multiple solutions to the first condition. Right afterwards, we will address the second one.

\begin{lma}
\label{Lemma existence splittings}
Let $F$ be an $A$-fixed component in $X_\w^\vi$ and fix a chamber $\mathfrak{C}$.
There are decompositions of graded vector spaces $V=V'\oplus V''$ and of graded $A_\w$-modules $W=W'\oplus W''$ of dimensions $\vi',\vi'',\w'$ and $\w''$ respectively and a parabolic subgroup $P\supset G_{\vi',\vi''}$ of $G_\vi$ that give a refinement $X_\w^{\vi',\vi''}$ of $X_\w^\vi$ such that the fixed component $F^{\vi',\vi''}\subset X_\w^{\vi',\vi''}$ determined by Lemma \ref{lemma choice componets above} lies in the image of the embedding $X_{\w'}^{\vi'}\times X_{\w''}^{\vi''}\subset X_\w^{\vi',\vi''}$.
\end{lma}
\begin{remq}
As the proof of the lemma will show, in general there are many different choices of splittings $V=V'\oplus V''$ and $W=W'\oplus W''$. Indeed, the proof shows that there are exactly $k-1$ possible splittings, where $k=\dim(W)$. This is important because the more the splittings are, the more will be the ways to produce the stable envelopes $\Stab{C}{F}$ of $X_\w^{\vi}$.
\end{remq}

\begin{proof}[Proof of the lemma]
Let $S$ be a maximal torus of $G_{\vi}$, and let $B\supset S$ be a Borel subgroup of $G_\vi$, with maximal nilpotent subalgebra $\mathfrak{n}=[\mathfrak{b},\mathfrak{b}]$.
Without loss of generality, we can assume that  $S=\prod_{i\in I}S_i$ is the standard diagonal maximal torus in $G_\vi=\prod_{i\in I}\GL(\vi_i)$ and $B=\prod_{i\in I}B_i$ is the Borel subgroup of upper triangular matrices:
\[
\left\{\begin{pmatrix}
    \ast& \ast & \dots & \ast \\
    0 &\ast & \dots & \ast \\
    \vdots & \vdots & \ddots & \vdots \\
    0 & 0 & \dots & \ast
  \end{pmatrix}\right\}_{i\in I}.
\]
We introduce the following notation: let $k=\dim{W}=\dim{\bigoplus_{i\in I} W_i}$ and let 
\[
W_i=\bigoplus_{j=1}^k \chi_{a_j}\otimes \Ci^{\w_i^{(j)}},
\]
where $a_j$ are coordinates of $A_\w\cong(\Ci^\times)^k$ and $\chi_{a_j}$ are the characters of $A_\w$ such that $\chi_{a_j}(a)=a_j$.
A compensating map of the $A$-fixed point $F\subset X_\w^{\vi}$ has the form 
\[
\psi_F:A\to G_{\vi}=\prod_{i\in I }\GL(\vi_i), \qquad \psi_F(a)=
\left\{\begin{pmatrix}
    a_1 \bm{1}_{\vi_i^{(1)}} & 0 & \dots & 0 \\
    0 &a_2\bm{1}_{\vi_i^{(2)}} & \dots & 0 \\
    \vdots & \vdots & \ddots & \vdots \\
    0 & 0 & \dots & a_k\bm{1}_{\vi_i^{(k)}}
  \end{pmatrix}\right\}_{i\in I}
\]
and gives weight decompositions
\[
V_i=\bigoplus_{j=1}^k \chi_{a_j}\otimes \Ci^{\vi_i^{(j)}}.
\]
These decompositions produce dimension vectors $\vi^{(j)}, \w^{(j)}\in \N^{I}$, for $j=1,\dots, k$ and we have
\[
F=\naka_Q(\vi^{(1)},\w^{(1)})\times \dots\times\naka_Q(\vi^{(k)},\w^{(k)})\hookrightarrow\naka_Q(\vi,\w).
\]
A generic chamber $\mathfrak{C}$ has the form 
\[
\mathfrak{C}=\lbrace a_{ \tau(1)}>\dots>a_{\tau(k)}\rbrace
\]
for some permutation $\tau\in \mathfrak{S}_k$.
The fixed point $F_S$ of the abelianization $X_S$ satisfying Lemma \ref{lemma choice componets above 2} has a compensating map of the form 
\[
\psi_{F_S}:A\to S=\prod_{i\in I }S_i, \qquad \psi_F(a)=
\left\{\begin{pmatrix}
    a_{ \tau(1)} \bm{1}_{{\vi_i}^{( \tau(1))}} & 0 & \dots & 0 \\
    0 &a_{ \tau(2)}\bm{1}_{{\vi_i}^{( \tau(2))}} & \dots & 0 \\
    \vdots & \vdots & \ddots & \vdots \\
    0 & 0 & \dots & a_{ \tau(k)}\bm{1}_{{\vi_i}^{( \tau(k))}}
  \end{pmatrix}\right\}_{i\in I}
\]
because in this way the weights of the subalgebra $\mathfrak{n}=[\mathfrak{b},\mathfrak{b}]$ are of the form
\begin{equation}
    \label{right condition example}
    \chi_{a_{\tau(i)}}\chi_{a_{\tau(j)}}^{-1} \qquad\text{for every $i< j$}
\end{equation}
and so there are no repelling weights with respect to $\mathfrak{C}$, as required.
Now let $1\leq k'\leq k$ and define
\[
V_i':=\bigoplus_{j=1}^{k'} \chi_{\tau(j)}\otimes \Ci^{\vi_i^{\tau(j)}}, \qquad V_i''=\bigoplus_{j=k'+1}^{k} \chi_{\tau(j)}\otimes \Ci^{\vi_i^{\tau(j)}}
\]
and similarly
\[
W_i':=\bigoplus_{j=1}^{k'} \chi_{\tau(j)}\otimes \Ci^{\w_i^{\tau(j)}} \qquad W_i''=\bigoplus_{j=k'+1}^{k} \chi_{\tau(j)}\otimes \Ci^{\w_i^{\tau(j)}}.
\]
Let $P$ be the parabolic subgroup of matrices of the form
\[
\left\{\begin{pmatrix}
    \GL(V_i')& \ast\\
    0 & \GL(V_i'')
  \end{pmatrix}\right\}_{i\in I}.
\]
We claim that the splittings above and the parabolic subgroup $P$ satisfy the conditions of the lemma.
Let $X_{\w'}^{\vi'}\times X_{\w''}^{\vi''}\hookrightarrow X_\w^{\vi',\vi''}$ be the embedding defined by the above splittings (cf. equation \eqref{embedding product naka}).
Because of \eqref{right condition example}, the fixed component $F^{\vi',\vi''}\subset X_\w^{\vi',\vi''}$ whose compensating map is the composition
\[
\begin{tikzcd}
\psi_{F^{\vi',\vi''}}:A\arrow[r, "\psi_{F_S}"]& S\arrow[r, hook]& G_{\vi',\vi''}=\prod_{i\in I}\GL(V_i')\times GL(V''_i)
\end{tikzcd}
\]
satisfies Lemma \ref{lemma choice componets above}. Moreover, by definition of compensating map, the preimage of $F^{\vi',\vi''}$ in $T^*\Repn{Q}{\vi}{\w}$ is contained in the subspace $T^*\Repn{Q}{\vi'}{\w'}\times T^*\Repn{Q}{\vi''}{\w''}\subset T^*\Repn{Q}{\vi}{\w}$, which implies that $F^{\vi',\vi''}$ is contained in $X^{\vi'}_{\w'}\times X^{\vi''}_{\w''}$.
\end{proof}
\begin{remq}
\label{compesating maps splitting fixed point}
Clearly, an $A$-fixed component $F^{\vi',\vi''}$ in $X_{\w'}^{\vi'}\times X_{\w''}^{\vi''
}\subset X^{\vi',\vi''}_\w$ is of the form $F^{\vi'}\times F^{\vi''}$, where $F^{\vi'}$ is an $A_{\w'}$-fixed component of $X^{\vi'}_{\w'}$ and $F^{\vi''}$ is an $A_{\w''}$-fixed component of $X^{\vi''}_{\w''}$. In particular, notice that compensating maps of the fixed component $F^{\vi',\vi''}=F^{\vi'}\times F^{\vi''}$ determined in the proof of the previous lemma are
\[
\psi_{F{\vi'}}(a)=\left\{\begin{pmatrix}
    a_{ \tau(1)} \bm{1}_{{\vi_i}^{( \tau(1))}} & 0 & \dots & 0 \\
    0 &a_{ \tau(2)}\bm{1}_{{\vi_i}^{( \tau(2))}} & \dots & 0 \\
    \vdots & \vdots & \ddots & \vdots \\
    0 & 0 & \dots & a_{ \tau(k')}\bm{1}_{{\vi_i}^{( \tau(k'))}}
  \end{pmatrix}\right\}_{i\in I}
\]
and
\[
\psi_{F^{\vi''}}(a)=\left\{\begin{pmatrix}
    a_{ \tau(k'+1)} \bm{1}_{{\vi_i}^{( \tau(k'+1))}} & 0 & \dots & 0 \\
    0 &a_{ \tau(k'+2)}\bm{1}_{{\vi_i}^{(\tau(k'+2))}} & \dots & 0 \\
    \vdots & \vdots & \ddots & \vdots \\
    0 & 0 & \dots & a_{ \tau(k)}\bm{1}_{{\vi_i}^{( \tau(k))}}
  \end{pmatrix}\right\}_{i\in I}.
\] 
\end{remq}
\begin{remq}
Inspecting the previous proof, one sees that the fixed components $F$, $F^{\vi',\vi''}$ and the fixed component $F_S$ defined in Lemma \ref{lemma choice componets above 2} fit in the following diagram:
\begin{equation*}
    \begin{tikzcd}
     & & F_S\cap\Fl_S\arrow[r, "(j_+^S)^A"]\arrow[d, "(\pi_S)^A"] & F_S\cap\widetilde\Fl_S \arrow[r, "(j_-^S)^A"]& F_S\\
    F^{\vi',\vi''}\cap\Fl\arrow[d, "\pi^A"]\arrow[r, "j_+^A"] & F^{\vi',\vi''}\cap\widetilde\Fl\arrow[r, "(j_-)^A"] & F^{\vi',\vi''}\\
    F
    \end{tikzcd}
\end{equation*}
where the top horizontal maps are built choosing $B\cap G_{\vi',\vi''}$ as Borel subgroup of $G_{\vi',\vi''}$.
\end{remq}

Lemma \ref{Lemma existence splittings} proves that 
the fixed component $F^{\vi',\vi''}$ associated to the pair $(F, \mathfrak{C})$ by Lemma \ref{lemma choice componets above} is contained in some subvariety of $X^{\vi',\vi''}_\w$ of the form $X^{\vi'}_{\w'}\times X^{\vi''}_{\w''}$, i.e. the existence of a commutative diagram of the form
\begin{equation*}
    \begin{tikzcd}
    F^{\vi',\vi''}\arrow[r,hookrightarrow] \arrow[d,hookrightarrow]& X_\w^{\vi',\vi''}\\
    X^{\vi'}_{\w'}\times X^{\vi''}_{\w''}\arrow[ru, hookrightarrow]
    \end{tikzcd}
\end{equation*}
Therefore, to solve the problem stated at the beginning of this section, it remains to check that the embeddings in the previous diagram satisfy Assumption \ref{Assumption main theorem}.

\begin{lma}
\label{Lemma Combinatioin conditions}
Let $F$ be an $A$ fixed component in $X^{\vi}_\w$ and $\mathfrak{C}$ a chamber. Let $P$ be a parabolic subgroup and $V=V'\oplus V''$, $W=W'\oplus W''$ two splittings satisfying the conditions of Lemma \ref{Lemma existence splittings}. 
Then the fixed component $F^{\vi',\vi''}\subset X^{\vi'}_{\w'}\times X^{\vi''}_{\w''}\subset  X^{\vi',\vi''}_{\w}$ produced by Lemma \ref{Lemma existence splittings} satisfies Assumption \ref{Assumption main theorem}.
\end{lma}
\begin{proof}
Recall that we have to show that the restriction of the normal bundle $\mathcal{N}_{\hat j}=\mathcal{N}^{X^{\vi',\vi''}_\w}_{\Att{\overline{C}}{X^{\vi'}_{\w'}\times X^{\vi''}_{\w''}}}$ to $F^{\vi',\vi''}$ fits in the following exact sequence:
\begin{equation*}
    \begin{tikzcd}
    0\arrow[r]& \restr{\mathcal{N}^{X_{\w'}^{\vi'}\times X_{\w''}^{\vi''}}_{\Att{C}{F^{\vi'}\times F^{\vi''}}}}{F^{\vi',\vi''}}\arrow[r]& \restr{\mathcal{N}^{X_{\w}^{\vi',\vi''}}_{\Att{C}{F^{\vi'}\times F^{\vi''}}}}{F^{\vi',\vi''}}\arrow[r]& \restr{\mathcal{N}_{\hat j}}{F^{\vi',\vi''}} \arrow[r] & 0
    \end{tikzcd}.
\end{equation*}
Since $F^{\vi',\vi''}\subset X^{\vi'}_{\w'}\times X^{\vi''}_{\w''}\subset  X^{\vi',\vi''}_{\w}$, we have the following exact sequence of bundles over $F^{\vi',\vi''}$:
\begin{equation*}
    \begin{tikzcd}
    0\arrow[r]& \mathcal{N}^{X^{\vi'}_{\w'}\times X^{\vi''}_{\w''}}_{F^{\vi',\vi''}}\arrow[r]& \mathcal{N}^{X^{\vi',\vi''}_\w}_{F^{\vi',\vi''}}\arrow[r]& \restr{\mathcal{N}_{X^{\vi'}_{\w'}\times X^{\vi''}_{\w''}}^{X^{\vi',\vi''}_\w}}{F^{\vi',\vi''}} \arrow[r] & 0
    \end{tikzcd}.
\end{equation*}
Since $F^{\vi',\vi''}$ is $A_\w$-fixed, this sequence splits in two exact sequence of bundles with only attracting or repelling weights with respect to $\mathfrak{C}$; in particular the one with repelling weights is
\begin{equation*}
    \begin{tikzcd}
    0\arrow[r]& \restr{\mathcal{N}^{X_{\w'}^{\vi'}\times X_{\w''}^{\vi''}}_{\Att{C}{F^{\vi'}\times F^{\vi''}}}}{F^{\vi',\vi''}}\arrow[r]& \restr{\mathcal{N}^{X_{\w}^{\vi',\vi''}}_{\Att{C}{F^{\vi'}\times F^{\vi''}}}}{F^{\vi',\vi''}}\arrow[r]& \restr{\mathcal{N}_{X^{\vi'}_{\w'}\times X^{\vi''}_{\w''}}^{X^{\vi',\vi''}_\w}}{F^{\vi',\vi''},<_\mathfrak{C}} \arrow[r] & 0
    \end{tikzcd}.
\end{equation*}
As a consequence, it suffices to show that
\begin{equation}
    \label{eqn too show}
    \restr{\mathcal{N}_{X^{\vi'}_{w'}\times X^{\vi''}_{\w''}}^{X^{\vi',\vi''}_\w}}{F^{\vi',\vi''},<_\mathfrak{C}} =\restr{\mathcal{N}_{\hat j}}{F^{\vi',\vi''}}
\end{equation}
as bundles over $F^{\vi',\vi''}$.
In addition, both these bundles are sub-bundles of 

\[\restr{\mathcal{N}_{X^{\vi'}_{w'}\times X^{\vi''}_{\w''}}^{X^{\vi',\vi''}_\w}}{F^{\vi',\vi''}},
\]
so it suffices to show that their fibers coincide.
First of all, notice that by the proof of Lemma \ref{Lemma existence splittings} the $A$-module $\Hom(W',W'')$ has no non-repelling weights. 
By Lemma \ref{lemma normal bundles}, the restriction of $\mathcal{N}_{X^{\vi'}_{\w'}\times X^{\vi''}_{\w''}}^{X^{\vi',\vi''}_\w}$ to a fixed point $p\in F^{\vi'}\times F^{\vi''}$ is the $A_\w$-module
\[
\mathcal{M}(V',V'',W',W'')+h^{-1}\mathcal{M}(V',V'',W',W'')^\vee,
\]
where $V'$ and $V''$ are acted on by $A$ via the compensating map $\psi_{F^{\vi',\vi''}}:A\to G_{\vi',\vi''}$. 
Since by assumption $F^{\vi',\vi''}$ is contained in $X^{\vi'}_{\w'}\times X^{\vi''}_{\w''}$, it follows that the map $\psi_{F_{\vi',\vi''}}$ acts on $V'$ with a subset of the weights with which it acts on $W'$, and similarly, it acts on $V''$ with a subset of the weights with which it acts on $W''$ (cf. Remark \ref{compesating maps splitting fixed point}). As a consequence, the $A$-module $\mathcal{M}(V',V'',W',W'')$ has no non-repelling weights too, because it is the direct sum of submodules of the form $\Hom(U',U'')$, where $U'$ is a submodule of $V'$ or $W'$ and $U''$ is a submodule of $V''$ or $W''$. Dually, $\mathcal{M}(V',V'',W',W'')^\vee$ has no non-attracting weights. Thus, we have that
\[
\left(\mathcal{N}_{X^{\vi'}_{w'}\times X^{\vi''}_{\w''}}^{X^{\vi',\vi''}_\w}\right)_{p,<_\mathfrak{C}}\cong\mathcal{M}(V',V'',W',W'')
\]
On the other hand, the subvariety $X^{\vi'}_{\w'}\times X^{\vi''}_{\w''}\subset X^{\vi',\vi''}_\w$ is the fixed component of the action of a subgroup $\Ci^\times\subset A_\w$ on $X_\w^{\vi',\vi''}$ that acts with weight one on $W'$ and trivially elsewhere and whose compensating map acts with weight one on $V'$ and trivially on $V''$. But the chamber $\mathfrak{\overline C}$ is the chamber $\lbrace a>0 \rbrace$, so the attracting sub-bundle of the normal bundle $\mathcal{N}_j$ has fibers isomorphic to  $\mathcal{M}(V',V'',W',W'')$, i.e. 
\[
\mathcal{N}_{\hat j,p}=\mathcal{M}(V',V'',W',W'').
\]
This shows equation \eqref{eqn too show}, and hence the lemma.
\end{proof}
Before stating the main result of the paper, we record one more lemma:
\begin{lma}
et $F$ be an $A$ fixed component in $X^{\vi}_\w$ and $\mathfrak{C}$ a chamber. Let $P$, $V=V'\oplus V''$ and $W=W'\oplus W''$ be a parabolic subgroups and two splittings satisfying the conditions of Lemma \ref{Lemma existence splittings} and let $F^{\vi',\vi''}\subset X^{\vi'}_{\w'}\times X^{\vi''}_{\w''}\subset  X^{\vi',\vi''}_{\w}$ be its associated fixed point.

Then the maps $j_+^A$ and $j_-^A$ in the diagram
\begin{equation*}
    \begin{tikzcd}
    F^{\vi',\vi''}\cap\Fl\arrow[d, "\pi^{A}"]\arrow[r, "j_+^{A}"] & F^{\vi',\vi''}\cap\Fl'^{A_{\w}}\arrow[r, "j_-^{A}"] & F^{\vi',\vi''}\\
    F
    \end{tikzcd}
\end{equation*}
are the identity map while $\pi^A$ is the isomorphism induced by the automorphism of $T^*\Repn{Q}{\vi}{\w}$ conjugating a representation by the inverse of the permutation $\tau\in \prod_{i\in I}\mathfrak{S}_{{\w_i}}$ defined in the proof of Lemma \ref{Lemma existence splittings}.
\end{lma}

\begin{proof}
The result follows by comparing the preimages of $F$ and $F^{\vi',\vi''}$ in $T^*\Repn{Q}{\vi}{\w}$.
\end{proof}

\subsection{The Shuffle product}
We can finally combine Theorem \ref{theorem abelianization} with Proposition \ref{ main proposition } to deduce the main result of this paper.
\begin{thm}
\label{shuffle theorem}
Let $\mathfrak{C}$ be a chamber and $F$ a fixed component of $X_\w^{\vi}$. Let also $P$, $V=V'\oplus V''$ and $W=W'\oplus W''$ be a parabolic subgroup and two splittings satisfying Lemma \ref{Lemma existence splittings}. Then  the stable envelope $\Stab{C}{F}$ with standard polarization $T^{1/2}X^{\vi}_\w$ is given by
\begin{equation}
\label{shuffle formula}
    \Stab{C}{F}=\pi! \circ j_+^* \circ (j_-!)^{-1}\circ \Theta(\mathcal{N}_{\hat j}) \tau^*(-h\delta)\left((\Stab{C'}{F^{\vi'}}\boxtimes \Stab{C''}{F^{\vi''}})^{lift}\right)\circ (\pi^A!)^{-1},
\end{equation}
where $F^{\vi'}\times F^{\vi''}\subset X^{\vi'}_{\w'}\times X^{\vi''}_{\w''}$ is the fixed component determined by Lemma \ref{Lemma existence splittings} and $\delta=\det \Delta\in \Char(G_{\vi',\vi''})$ is the character associated to the class $\Delta$ such that
\begin{equation}
    \label{equation shift}
    \Delta-h^{-1}\Delta^\vee=T^{1/2}X^{\vi',\vi''}_\w+h^{-1}\Nbund^\vee-\Nbund-T^{1/2}X^{\vi'}_{\w'}-T^{1/2}X^{\vi''}_{\w''}-\mathscr{N}_{\hat j}.
\end{equation}
More explicitly, we have
\begin{equation}
\label{explicit shuffle}
    \Stab{C}{F}=\Shuffle\restr{\left\lbrace\frac{\Theta(\mathscr{N}_{\hat j })\tau^*(-h\delta)\left((\Stab{C'}{F^{\vi'}}\boxtimes \Stab{C''}{F^{\vi''}})^{lift}\right)\circ (\pi^{A}!)^{-1}}{\Theta(\Nbund^\vee)\Theta(h^{-1}\Nbund^\vee)}\right\rbrace}{z'=z''=z}.
\end{equation}
Here, $\Shuffle=\prod_{i}\Shuffle^{(i)}$, and by $\Shuffle^{(i)}(g)$ we denote the function
\[
\Shuffle^{(i)}(g)(t^{(i)})=\sum_{\sigma^{(i)} \in \mathfrak{S}_{\vi'_i, \vi''_i}}g(\sigma^{(i)}\cdot t^{(i)}),
\]
where $\mathfrak{S}_{\vi'_i,\vi''_i}$ denotes the set of $(\vi'_i,\vi''_i)$-shuffles, permuting the $i$-th Chern roots 
\[
(t^{(i)}_1, \dots,t^{(i)}_{\vi'_i},t^{(i)}_{\vi'_i+1},\dots t^{(i)}_{\vi'_i+\vi_i''}),
\]
and the bundle $\Nbund$ is the bundle defined in Theorem \ref{maps relating naka and abel naka}.
\end{thm}
\begin{proof}
Formula \eqref{explicit shuffle} is simply obtained from \eqref{shuffle formula} by expressing the maps explicitly:
\begin{itemize}
    \item the pushforward $(j_-!)$ is multiplication by $\Theta(h^{-1}\Nbund^\vee)$,
    \item the pushforward $\pi!$ is 
    \[
    \pi!(f)=\Shuffle\left(\frac{f}{\Theta(\Nbund^\vee)}\right)
    \]
    \item the pullbacks are restrictions of the Chern roots of $X_\w^{\vi',\vi''}$ to the Chern roots of $X_\w^{\vi}$, so are omitted.
\end{itemize}
Finally, the shift $\delta$ is due to the fact that Theorem \ref{theorem abelianization} produces stable envelopes of $X_\w^{\vi}$ with standard polarization from stable envelopes of $X_\w^{\vi',\vi''}$ with polarization \[
T^{1/2}X_\w^{\vi',\vi''}+h^{-1}\Nbund^\vee-\Nbund
\] while Proposition \ref{ main proposition } produces stable envelopes of $X_\w^{\vi',\vi''}$ with the polarization defined in Lemma \ref{lemma normal bundles}: 
\[
T^{1/2}X_{\w'}^{\vi'}+T^{1/2}X_{\w''}^{\vi''}+\mathscr{N}_{\hat j},
\] so the result must be shifted by the determinant line bundle of $\Delta$, which is the class defined by equation \eqref{equation shift}, cf. \cite[Section 2.8.2]{aganagic2016elliptic}.
\end{proof}

\subsection{Inductive construction of stable envelopes}
Associated to every vertex $i$ of a quiver is the fundamental dimension vector $\delta_i\in \N^I$, defined as the dimension vector with a one in the $i$-th entry and zeros elsewhere. 
Clearly, every dimension vector can be written as an integral linear combination of the fundamental dimension vectors. Combining Theorem \ref{shuffle theorem} with the explicit formulas of the proof of Lemma \ref{Lemma existence splittings} we deduce the following:
\begin{cor}
Let $(Q,I)$ be a quiver. The elliptic stable envelopes of an arbitrary Nakajima variety $\naka_Q(\vi,\w)$ can be built inductively from the off-shell elliptic stable envelopes of the varieties $\naka_Q(\tilde\vi,\delta_i)$ for all $i\in I$ and $\tilde\vi\in \N^I$ by means of the shuffle product formula \eqref{explicit shuffle}.
\end{cor}


\section{Applications}

\subsection{Stable envelopes of Grassmannians}

We now apply Theorem \ref{shuffle theorem} to deduce an inductive formula for stable envelopes of the cotangent bundle of Grassmannian $T^*\Gr{v}{w}$. By Theorem \ref{char emb}, we have an embedding
\[
E_T(T^*\Gr{v}{w})\hookrightarrow E_T\times E^{(v)},
\]
induced by the elliptic characteristic class of the pullback by $T^*\Gr{v}{w}\to \Gr{v}{w}$ of the tautological bundle
\[
\V=\Set{(x,y)\in \Gr{v}{w}\times \Ci^w}{y\in x\subset \Ci^w}.
\]
We introduce coordinates $t_1,\dots, t_k$ on $E^{(v)}$ corresponding to the Chern roots of the bundle $\V$ and equivariant coordinates $a_1,\dots, a_w $ and $h$, which we denote collectively by $t$, $a$ and $h$.

Recall from Example \ref{fixed points Grassmannians} that fixed points of $T^*\Gr{v}{w}$ are labelled by subsets $\bm j_v=\lbrace j_1,\dots, j_v\rbrace \subset \lbrace 1,\dots, w\rbrace$, corresponding to $v$-dimensional coordinate planes in $W=\Ci^n$.

We choose the chamber 
\[
\mathfrak{C}_w=\left\{a_1>a_2>\dots>a_w\right\}
\]
and the standard polarization 
\[
T^{1/2}X=\Repn{Q}{\V}{\W}-\g_{\V}=\Hom(\V,\W)-\Hom(\V,\V).
\]
\begin{thm}
Let $w=w'+w''$. The off-shell elliptic stable envelopes $\text{Stab}^{v,w}_{\mathfrak{C}_w}(\bm j_v)$ of the cotangent bundle $T^*\Gr{v}{w}$ obey the following shuffle product formula:
\begin{multline*}
    \text{Stab}^{v,w}_{\mathfrak{C}_w}(\bm j_v)(t,a,z,h)=\\
    \Shuffle\left\{\phi(t,a,h)\text{Stab}^{v',w'}_{\mathfrak{C}_{w'}}(\bm j_{v'})(t',a',zh^{v''},h)\text{Stab}^{v'',w''}_{\mathfrak{C}_{w''}}(\bm j_{v''})(t'',a'',zh^{w'-v'},h)\right\}
\end{multline*}
with
\[
\phi(t,a,h)=\frac{\prod_{i=1}^{v'}\prod_{j=w'+1}^{w}\theta \left(\frac{a_j}{t_i}\right) \prod_{i=v'+1}^{v}\prod_{j=1}^{w'}\theta\left(\frac{t_i}{a_jh}\right)}{\prod_{i=1}^{v'}\prod_{j=v'+1}^{v}\theta\left(\frac{t_j}{t_i}\right) \theta\left(\frac{t_j}{t_ih}\right)}.
\]
and $\bm j_{v'}=\lbrace j_1,\dots, j_{v'}\rbrace\subset \lbrace1,\dots,w'\rbrace$, $\bm j_{v''}=\lbrace j_{v'+1},\dots, j_{v}\rbrace\subset \lbrace w'+1,\dots,w\rbrace$.
\end{thm}
This theorem together with Example \ref{off shell zero dim grass} allows to build explicit formulas for the stable envelopes of cotangent bundles of Grassmannians, which were already computed in \cite{aganagic2016elliptic} \cite{Felder2018} and \cite{rimanyi20193d}. In particular this formula is equal (up to different conventions, in particular the holomorphic normalization) to the shuffle product formula in \cite{Felder2018} that inspired this paper.
\begin{proof}[Proof of the theorem]
Construct the splittings $V=V'\oplus V''$ and $W=W'\oplus W''$ as in the proof of Lemma \ref{Lemma existence splittings}. By our choice of chamber, it follows that the permutation $\tau$ in that proof is simply the identity, so from Theorem \ref{shuffle theorem} we deduce that the stable envelope $\text{Stab}^{v,w}_{\mathfrak{C}_w}(\bm j_v)(t,a,z,h)$ is equal to
\[
\restr{\Shuffle\left\{\frac{\Theta(\mathscr{N}_{\hat j})\tau(-h\delta)^*\left(\text{Stab}^{v',w'}_{\mathfrak{C}_{w'}}(\bm j_{v'})(t',a',z',h)\text{Stab}^{v'',w''}_{\mathfrak{C}_{w''}}(\bm j_{v''})(t'',a'',z'',h)\right)}{\Theta(\mathfrak{N}^\vee)\Theta(h^{-1}\mathfrak{N}^\vee)}\right\}}{z'=z''=z}
\]
with $
\mathcal{N}_{\hat j}=\Hom(\V',\W'')+h^{-1}\Hom(\W',\V'')
$
and 
$
\mathfrak{N}=\Hom(\V'',\V')
$, 
which in terms of the elliptic Chern roots can be written as  
\begin{equation*}
    \mathcal{N}_{\hat j}=\sum_{i=1}^{v'}\sum_{j=w'+1}^{w}\frac{a_j}{t_i}+ \sum_{i=v'+1}^{v}\sum_{j=1}^{w'}\frac{t_i}{a_jh} \qquad \mathfrak{N}=\sum_{i=1}^{v'}\sum_{j=v'+1}^{v}\frac{t_i}{t_j}.
\end{equation*}
Hence, we have that 
\[
\frac{\Theta(\mathcal{N_<})}{\Theta(\mathfrak{N}^\vee)\Theta(h^{-1}\mathfrak{N}^\vee)}=\phi(t,a,h).
\]
Thus, to prove the theorem, it suffice to compute the $h$-shifts of $z'$ and $z''$ determined by $\tau(-h\delta)^*$. 
Since 
\[
 T^{1/2}X^{\vi',\vi''}_{\w}-\mathcal{N}_{\hat j}-T^{1/2}X^{\vi'}_{\w'}-T^{1/2}X^{\vi''}_{\w''}=\Hom(\V'',\W')-h^{-1}\Hom(\W',\V'')
 \]
it follows that 
\begin{align*}
    \Delta-h^{-1}\Delta^\vee=\Hom(\V'',\W')-h^{-1}\Hom(\W',\V'')+h^{-1}\mathfrak{N}^\vee-\mathfrak{N}
\end{align*}
and hence $\Delta=\Hom(\V'',\W')-\mathfrak{N}$, which in turn corresponds to the shifts $z'\mapsto z'h^{v''}$ and $z''\mapsto z''h^{w'-v'}$.
\end{proof}

\subsection{Stable envelopes of the instanton moduli space}
\label{subsection Stable envelopes of the instanton moduli space}

Another classical Nakajima variety is the instanton moduli space $\mathcal{M}(v,w)$. In Example \ref{Jordan quiver} we have recalled that, in accordance with the general theory, this variety admits an action of the torus $T=A_{\w}\times B\times \Ci^\times_h$, with $A_\w=(\Ci^\times)^w$ and $B\cong \Ci^\times$. As a consequence, this $A$-variety lies beyond the realm of validity of Theorem \ref{shuffle theorem}, where no $B$-action is considered, i.e. when $A=A_\w$. Technically speaking, this is due to the fact that Lemma \ref{Lemma existence splittings} and Lemma \ref{Lemma Combinatioin conditions} are in general false when $A\supsetneq A_\w$. However, for certain choices of chambers $\mathfrak{C}$, these lemmas still hold.

In this last section, we study the stable envelopes of the instanton moduli space $\mathcal{M}(v,w)$ for such special chambers. Combined with Smirnov's explicit formulas for the Hilbert scheme of points \cite{smirnov2018elliptic}, we produce explicit inductive formulas for the stable envelopes of $\mathcal{M}(v,w)$.

As in the last section, by Theorem \ref{char emb} we have an embedding
\[
E_T(\mathcal{M}(v,w))\hookrightarrow E_T\times E^{(v)},
\]
induced by the elliptic characteristic class of the tautological bundle $\V$ on $\mathcal{M}(v,w)$.
We introduce coordinates $t_1,\dots, t_k$ on $E^{(v)}$ corresponding to the Chern roots of the bundle $\V$ and equivariant coordinates $a_1,\dots, a_w,b$ and $h$, which we denote collectively by $t$, $a$, $b$ and $h$. Passing to a double cover if needed, we also define $r_1=bh^{-1/2}$ and  $r_2=bh^{-1/2}$, so that $r_1r_2=h^{-1}$. These are the equivariant parameters in terms of which the stable envelopes of the Hilbert scheme are computed in \cite{smirnov2018elliptic}, to which we refer for a detailed description of the action in terms of these variables\footnote{In \cite{smirnov2018elliptic}, the indeterminates $x$, $a$ and $t$ are used in place of our $t$ $b$ and $r$. Moreover, the variable $h$ is replaced with its inverse $h^{-1}$.}. 

As for the polarization, we choose
\[
T^{1/2}\mathcal{M}(v,w)=\Hom(\W,\V) +r_1\Hom(\V,\V)-\Hom(\V,\V).
\]
As we have already mentioned, the choice of chamber is delicate; consider the chambers 
\[
\mathfrak{C}_{w,\tau}=\left\{a_{\tau(1)}>a_{\tau(2)}> \dots > a_{\tau(w)}, \frac{a_{\tau(i)}}{a_{\tau(i+1)}}\gg b \text{ for $i=1,\dots,w-1$} \right\}.
\]
Here $\frac{a_{\tau(i)}}{a_{\tau(i+1)}}\gg b$ means that we choose a chamber such that the weights $\frac{a_{\tau(i)}}{a_{\tau(i+1)}}$ are grater than any power of $b$ that appears as a weight of the normal bundle of a fixed point in $\mathcal{M}(v,w)$.
\begin{lma}
\label{lemma exception instantons}
For these choices of chambers $\mathfrak{C}_{w,\tau}$, Lemma \ref{Lemma existence splittings} and Lemma \ref{Lemma Combinatioin conditions} hold.
\end{lma}
\begin{proof}
By our choice of the chamber, a weight $\frac{a_{i}}{a_{j}}b^k$ is attracting if and only if $\frac{a_{i}}{a_{j}}$ is attracting. As a consequence, a direction in the normal bundle of an $A_\w\times B$-fixed point is attracting with respect to $\mathfrak{C}_{w,\tau}$ if and only if it is attracting with respect to the action of the subgroup $A_\w$ and the chamber $\left\{a_{\tau(1)}>a_{\tau(2)}> \dots > a_{\tau(w)} \right\}$. This implies that the proofs of the lemmas hold without changes.
\end{proof}
\begin{remq}
Geometrically, these choices of chambers imply that if two $A$-fixed points $F_1$ and $F_2$ are contained in two different $A_\w$-fixed components $Y_1$ and $Y_2$, then $\text{Att}^f_{\mathfrak{C}_{w,\tau}}(F_1)\cap F_2\neq \emptyset$ if and only if $\text{Att}^f_{\widetilde{\mathfrak{C}}_{w,\tau}}(Y_1)\cap Y_2\neq \emptyset$, where $\widetilde{\mathfrak{C}}_{w,\tau}$ is the restriction of $\mathfrak{C}_{w,\tau}$ to $A_\w$. This is because $\text{Att}^f_{\mathfrak{C}_{w,\tau}}(F_1)\cap F_2\neq \emptyset$ if and only if there exists a direction of the normal bundle of $F_1$ in $\mathcal{M}(v,w)$ pointing towards $F_2$ (and hence in the normal bundle of $Y_1$, because $F_2\not\in Y_1$) which is an $A_\w\times B$-module with attracting weight. But the normal weights are of the form $\frac{a_i}{a_j}b^k$, so for these choices of chambers such a weight is attracting if and only if the weight $\frac{a_i}{a_j}$ is attracting. This means that $\text{Att}^f_{\mathfrak{C}_{w,\tau}}(F_1)\cap F_2\neq \emptyset$ if and only if there is a a direction in the normal bundle of $Y_1$ pointing towards $Y_2$ which is an attracting $A_\w$-weight, i.e. iff $\text{Att}^f_{\widetilde{\mathfrak{C}}_{w,\tau}}(Y_1)\cap Y_2\neq \emptyset$.
\end{remq}
The previous lemma implies that, for these choices of chambers, Theorem \ref{shuffle theorem} still holds in the setting of this section. Applying it now for the choice of chamber $\mathfrak{C}_w=\mathfrak{C}_{w,id}$, we get
\begin{thm}
Let $w=w'+w''$. The off-shell elliptic stable envelopes $\text{Stab}^{v,w}_{\mathfrak{C}_w}(\bm\lambda )$ of the instanton moduli spaces $\mathcal{M}(v,w)$ obey the following shuffle product formula:
\begin{multline*}
    \text{Stab}^{v,w}_{\mathfrak{C}_w}(\bm\lambda )(t,a,r,z,h)=\\
    \Shuffle\left\{\phi(t,a,r,h)\text{Stab}^{v',w'}_{\mathfrak{C}_{w'}}(\bm\lambda')(t',a',r,zh^{-w''},h)\text{Stab}^{v'',w''}_{\mathfrak{C}_{w''}}(\bm\lambda'')(t'',a'',r,z,h)\right\}
\end{multline*}
with
\[
\phi(t,a,r,h)=\frac{\prod_{i=v'+1}^{v}\prod_{j=1}^{w'}\theta\left(\frac{t_i}{a_j}\right)\prod_{i=1}^{v'}\prod_{j=w'+1}^{w} \theta\left(\frac{a_j}{ht_i}\right)\prod_{i=1}^{v'}\prod_{j=v'+1}^{v}\theta\left(\frac{r_1t_j}{t_i}\right)\theta\left(\frac{r_2t_j}{t_i}\right)}{\prod_{i=1}^{v'}\prod_{j=v'+1}^{v}\theta\left(\frac{t_j}{t_i}\right)\theta\left(\frac{t_j}{t_ih}\right)}
\]
$\bm\lambda'=(\lambda^{(1)},\dots, \lambda^{(w')})$, $
\bm\lambda''=(\lambda^{(w'+1)},\dots, \lambda^{(w)})$, $v'=\sum_{j=1}^{w'}\sum_m \lambda_m^{(j)}$ and $v''=v-v'$.
\end{thm}
Together with the formula for the elliptic stable envelopes for the Hilbert scheme $\Hilb_v=\mathcal{M}(v,1)$ in \cite{smirnov2018elliptic}, this result allows to produce explicit formulas for the stable envelopes of the instanton moduli space $\mathcal{M}(v,w)$.

\begin{proof}[Proof of the theorem]
As noticed in the proof of Lemma \ref{lemma exception instantons} we can retain the construction of the proof of Lemma \ref{Lemma existence splittings}.
As a consequence, applying Theorem \ref{shuffle theorem} we deduce that the stable envelope $\text{Stab}^{v,w}_{\mathfrak{C}_w}(\bm\lambda )(t,a,r,z,h)$ is equal to
\[
\restr{\Shuffle\left\{\frac{\Theta(\mathscr{N}_{\hat j})\tau(-h\delta)^*\left(\text{Stab}^{v',w'}_{\mathfrak{C}_{w'}}(\bm \lambda')(t',a',r,z',h)\text{Stab}^{v'',w''}_{\mathfrak{C}_{w''}}(\bm \lambda'')(t'',a'',r,z'',h)\right)}{\Theta(\mathfrak{N}^\vee)\Theta(h^{-1}\mathfrak{N}^\vee)}\right\}}{z'=z''=z}
\]
with
\[
\mathcal{N}_{\hat j}=r_1\Hom(\V',\V'')+\Hom(\W',\V'')+r_2\Hom(\V',\V'')+h^{-1}\Hom(\V',\W'')
\]
and 
$
\mathfrak{N}=\Hom(\V'',\V').
$
By multiplicativity of the Thom class, we have that 
\begin{align*}
\frac{\Theta(\mathcal{N}_{\hat j})}{\Theta(\mathfrak{N}^\vee)\Theta(h^{-1}\mathfrak{N}^\vee)}&=\frac{\Theta(\Hom(\W',\V'')+h^{-1}\Hom(\V',\W''))\Theta(r_1\Hom(\V',\V''))\Theta(r_2\Hom(\V',\V''))}{\Theta(\Hom(\V',\V''))\Theta(h^{-1}\Hom(\V',\V''))}\\&= \phi(t,a,r,h),
\end{align*}
where in the last step we have used the following identities
\begin{equation*}
\Hom(\W',\V'')+h^{-1}\Hom(\V',\W'')= \sum_{i=v'+1}^{v}\sum_{j=1}^{w'}\frac{t_i}{a_j}+\sum_{i=1}^{v'}\sum_{j=w'+1}^{w}\frac{a_j}{ht_i},
\end{equation*}
\begin{equation*}
	\Hom(\V',\V'')=\sum_{i=1}^{v'}\sum_{j=v'+1}^{v} \frac{t_j}{t_i}.
\end{equation*}
Thus, to prove the theorem, it suffices to compute the $h$-shifts of $z'$ and $z''$ determined by equation \eqref{equation shift}. 
Since in $A\times\Ci_h^\times$-equivariant K-theory we have
\[
 T^{1/2}X^{\vi',\vi''}_{\w}-\mathcal{N}_{\hat j}-T^{1/2}X^{\vi'}_{\w'}-T^{1/2}X^{\vi''}_{\w''}=\Hom(\W'',\V')-h^{-1}\Hom(\V',\W'')+\mathfrak{N}-h^{-1}\mathfrak{N}^\vee
 \]
it follows that 
\begin{align*}
    \Delta-h^{-1}\Delta^\vee=\Hom(\W'',\V')-h^{-1}\Hom(\V',\W'')
\end{align*}
and hence $\Delta=\Hom(\W'',\V')$, which in turn corresponds to the shifts $z'\mapsto zh^{-w''} $ and $z''\mapsto z''$.

\end{proof}

\section*{Appendix: Proof of Proposition \ref{ main proposition }}
\label{appendix}

\textbf{Proposition 5.3.}
Let $F^{\vi'}$ and $F^{\vi''}$ be fixed components in $X_{\w'}^{\vi'}$ and $X_{\w''}^{\vi''}$ respectively, and let $\mathfrak{C}'$ and $\mathfrak{C}''$ be the chambers in $\text{Cochar}(A')\otimes_\Z\R$ and $\text{Cochar}(A'')\otimes_\Z\R$ induced by the inclusions $A'\hookrightarrow A$ and $A''\hookrightarrow A$ respectively.
Assume also the following:

\noindent\textbf{Assumption 5.4}
The restriction of the normal bundle $\mathscr{N}_{\hat j}$ of the attracting manifold 
\[
\hat j:\Att{\overline{C}}{X_{\w'}^{\vi'}\times X_{\w''}^{\vi''}}\hookrightarrow X^{\vi',\vi''}_\w
\]
fits in the following exact sequence of bundles over $F^{\vi'}\times F^{\vi''}$:
\begin{equation*}
    \begin{tikzcd}
    0\arrow[r]& \restr{\mathcal{N}^{X_{\w'}^{\vi'}\times X_{\w''}^{\vi''}}_{\Att{C}{F^{\vi'}\times F^{\vi''}}}}{F^{\vi'}\times F^{\vi''}}\arrow[r]& \restr{\mathcal{N}^{X_{\w}^{\vi',\vi''}}_{\Att{C}{F^{\vi'}\times F^{\vi''}}}}{F^{\vi'}\times F^{\vi''}}\arrow[r]& \restr{\mathcal{N}_{\hat j}}{F^{\vi'}\times F^{\vi''}} \arrow[r] & 0
    \end{tikzcd},
\end{equation*}
where $\mathcal{N}^{X_{\w'}^{\vi'}\times X_{\w''}^{\vi''}}_{\Att{C}{F^{\vi'}\times F^{\vi''}}}$ is the normal bundle of $\Att{C}{F^{\vi'}\times F^{\vi''}}$ in $X_{\w'}^{\vi'}\times X_{\w''}^{\vi''}$ and $\mathcal{N}^{X_{\w}^{\vi',\vi''}}_{\Att{C}{F^{\vi'}\times F^{\vi''}}}$ is the normal bundle of $\Att{C}{F^{\vi'}\times F^{\vi''}}$ in $X_{\w}^{\vi',\vi''}$.
\noindent
Then for any choice of lift $\Stab{C'}{F^{\vi'}}\boxtimes\Stab{C''}{F^{\vi''}})^{lift}$ we have 
\begin{equation}
\label{formula stab refinement}
    \Stab{C}{F^{\vi'}\times F^{\vi''}}=\Theta(\mathcal{N}_{\hat j})(\Stab{C'}{F^{\vi'}}\boxtimes \Stab{C''}{F^{\vi''}})^{lift},
\end{equation}
where $\Stab{C}{F^{\vi'}\times F^{\vi''}}$ is the stable envelope of the fixed component $F^{\vi'}\times F^{\vi''}\subset X_\w^{\vi',\vi''}$ with polarization
\[
T^{1/2}X_{\w'}^{\vi'}+T^{1/2}X_{\w''}^{\vi''}+\mathcal{N}_{\hat j},
\]
defined in Lemma \ref{lemma normal bundles} and Remark \ref{remark extension bundles}.

We give two proofs of the theorem, each one enlightening different aspects of the theory. The second one in particular focuses on the case when the lift $(\Stab{C'}{F^{\vi'}}\boxtimes \Stab{C''}{F^{\vi''}})^{lift}$ is found by abelianization of the stable envelopes $\Stab{C'}{F^{\vi'}}$ and $\Stab{C''}{F^{\vi''}}$, as suggested in Remark \ref{remark existence lifts}.

\begin{proof}[First proof]$ $
\newline
\textbf{Step 1:}
For brevity, let us define $F:=F^{\vi'}\times F^{\vi''}$, let $\ind^{\vi',\vi''}$ be the index of $X^{\vi',\vi''}_\w$ and let $\ind^\vi$ be the index of $X^\vi_\w$.
Firstly, we check that both sides of equation \eqref{formula stab refinement} are morphisms between the same sheaves.
The left hand side is a morphism between the following sheaves over the base $\Base{X_\w^{\vi',\vi''}}{T}$:
\begin{equation}
\label{right sheaves}
    \Theta(T^{1/2}F)\otimes \tau(-h\det\ind^{\vi',\vi''})^*\mathscr{U}_{F}\to \Theta(T^{1/2}X_\w^{\vi',\vi''})\otimes \mathscr{U}_{X_\w^{\vi',\vi''}}\otimes\Theta(h)^{\rk(\ind^{\vi',\vi''})}.
\end{equation}
Explicitly, $\mathscr{U}_{X_\w^{\vi',\vi''}}$ has the automorphy of the function
\[
\prod_{i\in I}\phi(\det\aV',z_i')\phi(\det\aV'',z_i'')\prod_{j }\phi(a_j,z_{a_j})\phi(h,z_h),
\]
and thus by Lemma \ref{automorphy translate universal} $\tau(-h\det\ind^{\vi',\vi''})^*\mathscr{U}_{F}$ has the automorphy of 
\begin{align*}
\prod_{i\in I}\phi(\restr{\det\aV'}{F},z_i')\phi(\restr{\aV''}{F},z_i'')\times
  &\times \phi(h^{-1},\det\ind^{\vi',\vi''})\prod_{j}\phi(a_j,z_{a_j})\phi(h,z_h).  
\end{align*}
Now notice that $\ind^{\vi',\vi''}=\restr{T^{1/2}X_\w^{\vi',\vi''}}{F,>}=\restr{T^{1/2}X_{\w'}^{\vi'}}{F,>}+\restr{T^{1/2}X_{\w''}^{\vi''}}{F,>}=\ind^{\vi'}+\ind^{\vi''}$
because $\restr{\mathscr{N}_{\hat j}}{F,>}=0$ by Lemma \ref{lemma normal bundles}.
As a consequence, it is easy to see that the lift
\begin{equation}
\label{shuffle stab}
    (\Stab{C'}{F^{\vi'}}\boxtimes \Stab{C''}{F^{\vi''}} )^{lift}
\end{equation}
is a morphism between the following sheaves over $\Base{X_\w^{\vi',\vi''}}{T}$:
\begin{align*}
    \Theta(T^{1/2}F)\otimes \tau(-h\det\ind^{\vi',\vi''})^*\mathscr{U}_{F}\to
    \Theta(T^{1/2}X_{\w'}^{\vi'}+T^{1/2}X_{\w''}^{\vi''})\otimes  \mathscr{U}_{X_\w^{\vi',\vi''}}\otimes\Theta(h)^{\rk(\ind^{\vi',\vi''})}.
\end{align*}
Moreover, by our choice of polarization the Thom sheaf $\Theta(T^{1/2}X^{\vi'}_{\w'}+T^{1/2}X^{\vi''}_{\vi''})$ is isomorphic to $\Theta(T^{1/2}X^{\vi',\vi''}_\w-\mathscr{N}_{\hat j})$, so the composition of \eqref{shuffle stab} with multiplication by $\Theta(\mathcal{N}_{\hat j})$ is well defined and gives a morphism between sheaves as in \eqref{right sheaves}.
\newline
\textbf{Step 2: }
Let $\Ci^\times $ be the torus fixing $X^{\vi'}_{\w'}\times X^{\vi''}_{\w''}\subset X^{\vi',\vi''}_{\w}$ as in Remark \ref{Emebdded nakajima varieties as torus-fixed component}, and let $\overline {\mathfrak{C}}$ be the projection of $\mathfrak{C}$ on $\text{Cochar}(\Ci^\times)\otimes_\Z\R$.
By definition of the lift \eqref{shuffle stab} and the assumption of the proposition, it follows that \eqref{formula stab refinement} satisfies the axiom \ref{condition 2 stable env} of stable envelopes, so it suffices to show that it satisfies axiom \ref{Condition 1 stable env} as well. Since the lift \eqref{shuffle formula} has the same support as the stable envelope of $F$ in $X^{\vi'}_{\w'}\times X^{\vi''}_{\w''}$, it suffices to show that the support of the class $\Theta(\mathcal{N}_{\hat j})$ has the same support as $\text{Stab}_{\overline{\mathfrak{C}}}{(X^{\vi'}_{\w'}\times X^{\vi''}_{\w''})}$, i.e. that for every component $Y\subset (X^{\vi',\vi''}_{\w})^{\Ci^\times}$ such that $Y\cap \text{Att}^f_{\overline{\mathfrak{C}}}(X_{\w''}^{\vi'}\times X_{\w''}^{\vi''})=\emptyset$, then $i_Y^*(\Theta(\mathcal{N}_{\hat j}))=0$. Here $i_Y$ is the inclusion map $Y\hookrightarrow X^{\vi',\vi''}_\w$.

Set $E=\text{Att}_{\overline{\mathfrak{C}}}(X_{\w''}^{\vi'}\times X_{\w''}^{\vi''})$, $E^f=\text{Att}^f_{\overline{\mathfrak{C}}}(X_{\w''}^{\vi'}\times X_{\w''}^{\vi''})$ and define $\widetilde X^{\vi',\vi''}_\w=X_\w^{\vi',\vi''}\setminus(E^f\setminus E)$.
Consider the commutative diagram

\begin{equation}
\label{diagram normal}
  \begin{tikzcd}
X_{\w'}^{\vi'}\times X_{\w''}^{\vi''} \arrow[d, swap, bend right, hookrightarrow] \arrow[r,"j", hookrightarrow] & X_\w^{\vi',\vi''}\\
E\arrow[r, hookrightarrow, "\tilde{j}"]\arrow[ur, "\hat j", hookrightarrow]\arrow[u, bend right, swap, "\rho"]  & \widetilde X^{\vi',\vi''}_\w\arrow[u, hookrightarrow]
\end{tikzcd}  
\end{equation}
where
\[
\rho:E\to X_{\w'}^{\vi'}\times X_{\w''}^{\vi''}
\]
sends a point $x\in E$ to $\lim_{z\to 0}\sigma(z)\cdot x\in X_{\w'}^{\vi'}\times X_{\w''}^{\vi''}$, for any (every) $\sigma\in \overline{\mathfrak{C}}$.
Notice that the embedding of $\hat j: E \to X^{\vi',\vi''}_\w$ is not closed but $\tilde j: E \to \widetilde X^{\vi',\vi''}_\w$ is a closed embedding, hence proper. In particular, the composition $\tilde j!\circ \rho^*!$ is well defined, and is given by multiplication by $\Theta(\mathcal{N}_{\hat j})$, the Thom class of the normal bundle of the embedding $\tilde j$. 
Let $k:\widetilde X^{\vi',\vi''}_\w\setminus E\hookrightarrow X^{\vi',\vi''}_\w$ denote the inclusion. By the long exact sequence in elliptic cohomology \cite{ginzburg1995elliptic}
\begin{equation*}
\begin{tikzcd}
0 \arrow[r]& \Theta(-\mathcal{N}_{\hat j})\arrow[r, "\tilde{j}!\circ \rho^*"]& \mathcal{O}_{E_T(\widetilde X^{\vi',\vi''}_\w)}\arrow[r, "k^*"] &\mathcal{O}_{E_T(\widetilde X^{\vi',\vi''}_\w\setminus E)} \arrow[r] & \dots
\end{tikzcd}
\end{equation*}
it follows that the composition $k^*\circ \tilde{j}!\circ \rho^*$ is zero.
As a consequence, if $Y\subset (X^{\vi',\vi''}_{\w})^{\Ci^\times}$ is such that $Y\cap E^f=\emptyset$, then the inclusion $i_Y$ factors as
\begin{equation*}
    \begin{tikzcd}
Y\arrow[d, hookrightarrow, "\tilde i_Y"]\arrow[r, hookrightarrow, "i_Y"]& X^{\vi',\vi''}_\w\\
\widetilde X^{\vi',\vi''}_\w\setminus E \arrow[ru, hookrightarrow, "k"]
    \end{tikzcd}
\end{equation*}
Therefore, we have that \[
i_Y^*\Theta(\mathcal{N}_{\hat j})=\tilde{i}_Y^* \circ k^* \Theta(\mathcal{N}_{\hat j})=\tilde{i}_Y^*\circ  k^*\circ \tilde{j}!\circ \rho^*=0
\]
\end{proof}

\begin{proof}[Second proof] $ $

\noindent\textbf{Step 1:}
In this step we turn the claim of the theorem into an equivalent statement. 

Let $S=S'\times S''\subset G_{\vi',\vi''}\cong G_{\vi'}\times G_{\vi''}$ be a maximal torus. The same arguments of Section \ref{subsection embedding nakajima varieties} allow to define the following inclusion of abelianized varieties
\[
j_{S}:X_{S'}\times X_{S''}\hookrightarrow X_S,
\]
obtained by quotienting the $\theta$-semistable locus of the inclusion 
\[
\mu_{S'}^{-1}(0)\times\mu_{S''}^{-1}(0)\hookrightarrow\mu_{S}^{-1}(0).
\]
Moreover, given a Borel subgroup $B=B'\times B''\subset G_{\vi',\vi''}\cong G_{\vi'}\times G_{\vi''}$ containing $S$ and a Borel subgroup $B\subset G_{\vi}$ containing $B$ we can define closed inclusions 
\begin{equation*}
    \begin{tikzcd}[]
    & \mu^{-1}(0)\arrow[r, hookrightarrow]& \mu^{-1}(\mathfrak{b}^\perp)\arrow[r, hookrightarrow]&\mu_S^{-1}(0)\\ \mu_{\vi'}^{-1}(0)\times \mu_{\vi''}^{-1}(0) \arrow[ur, hookrightarrow]\arrow[r, hookrightarrow] &\mu_{\vi'}^{-1}((\mathfrak{b}')^\perp)\times \mu_{\vi''}^{-1}((\mathfrak{b}'')^\perp) \arrow[ur, hookrightarrow]\arrow[r, hookrightarrow]&\mu_{S'}^{-1}(0)\times \mu_{S''}^{-1}(0) 
    \arrow[ur, hookrightarrow]
    \end{tikzcd}
\end{equation*}
This diagram restricts to the $\theta$-semistable locus and, taking quotients by $S=S'\times S''$ and $G_{\vi',\vi''}$, descends to the following diagram of $T$-varieties: 

\begin{equation*}
    \begin{tikzcd}[column sep=huge]
    & \Fl_{S}\arrow[d, dash]\arrow[r, hookrightarrow, "j_+"]&\widetilde\Fl_{S}\arrow[r, hookrightarrow, "j_-"]&X_S\\ \Fl'\times \Fl''\arrow[dd, "\pi'\times \pi''"] \arrow[ur, hookrightarrow]\arrow[r, hookrightarrow, "j_+'\times j_+''"] &\widetilde\Fl'\times \widetilde\Fl'' \arrow[d, "\pi"] \arrow[ur, hookrightarrow]\arrow[r, hookrightarrow, "j_-'\times j_-''"]&X_{S'}\times X_{S''} \arrow[ur, hookrightarrow]
    \\
    & X_{\w'}^{\vi'\vi''}
    \\
    X_{\w'}^{\vi'}\times X_{\w''}^{\vi''}\arrow[ur]
    \end{tikzcd}
\end{equation*}
Now let $F_S$ be the fixed component in $X_S$ that satisfies the conditions of Lemma \ref{lemma choice componets above 2}. Then the previous diagram can be restricted as follows:
\begin{equation*}
    \begin{tikzcd}[column sep=huge]
    & \Fl^{S}\cap F_S\arrow[d, dash]\arrow[r, hookrightarrow, "(j_+)^{A}"]&\Fl^{S}\cap F_S\arrow[r, hookrightarrow, "(j_-)^{A}"]& F_S\\ \Fl^{\vi'}\times \Fl^{\vi''}\cap F_{S'}\times F_{S''} \arrow[dd, "(\pi'\times \pi'')^{A}"] \arrow[ur, hookrightarrow]\arrow[r, hookrightarrow, "(j_+'\times j_+'')^{A}"] & Fl'\times Fl^{\vi''}\cap F_{S'}\times F_{S''}  \arrow[d, "\pi^{A}"] \arrow[ur, hookrightarrow]\arrow[r, hookrightarrow, "(j_-'\times j_-'')^{A}"]& F_{S'}\times F_{S''} \arrow[ur, equal]
    \\
    & F
    \\
    F^{\vi'}\times F^{\vi''}\arrow[ur,equal]
    \end{tikzcd}
\end{equation*}
The equality $F_{S'}\times F_{S''}=F_S$ can be checked by noticing that $F_{S'}\times F_{S''}$ is a fixed component of $X_S$ and satisfies the conditions of Lemma \ref{lemma choice componets above 2}, and hence by uniqueness it must be equal to $F_S$.
Now notice that we have a commuting diagram 
\begin{equation*}
    \begin{tikzcd}
    \Pic{T}{X_{\w}^{\vi',\vi''}}\arrow[r] \arrow[d] &  \Pic{T}{X_S}\arrow[d]\\
    \Pic{T'}{X_{\w'}^{\vi'}}\times \Pic{T''}{X_{\w''}^{\vi''}}\arrow[r] &\Pic{T'}{X_{S'}}\times \Pic{T''}{X_{S''}},
    \end{tikzcd}
\end{equation*}
where the left vertical map is the one in diagram \eqref{map picards exterior tenso} and the left vertical one is built analogously.
This diagram can be tensored with $E$ to give 
\begin{equation*}
    \begin{tikzcd}
    \Epic{T}{X_{\w}^{\vi',\vi''}}\arrow[r, "\lambda"] \arrow[d, "\psi"] &  \Epic{T}{X_S}\arrow[d, "\psi_S"]\\
    \Epic{T'}{X_{\w'}^{\vi'}}\times \Epic{T''}{X_{\w''}^{\vi''}}\arrow[r,"\lambda'\times \lambda''"] &\Epic{T'}{X_{S'}}\times \Epic{T''}{X_{S''}}.
    \end{tikzcd}
\end{equation*}
Stable envelope of the Nakajima varieties $X_{\w'}^{\vi'}$ and $X_{\w''}^{\vi''}$ make the following diagrams commute:
\begin{equation*}
\label{stable envelopes }
    \begin{tikzcd}[column sep=huge]
    \Theta(T^{1/2}F^{\vi'})\otimes \tau^*\mathscr{U}_{(X_{\w'}^{\vi'})^{A}}\arrow[rr, "(j'_-)^{A}!\circ (((j_+')^{A})^*)^{-1} \circ(\pi')^{A}!)^{-1}"]\arrow[d, "\Stab{C'}{F^{\vi'}}"]& & \Theta(T^{1/2}F_{S'})\otimes \tau^*\mathscr{U}_{X_{S'}^{A}} \arrow[d, "\lambda'^*\Stab{C'}{F_{S'}}"]\\
    \Theta(T^{1/2}X_{\w'}^{\vi'})\otimes \mathscr{U}_{X_{\w'}^{\vi'}}\otimes
    \Theta(h)^{\rk(\ind^{\vi'})}.
    & & \Theta(T^{1/2}X_{S'})\otimes \mathscr{U}_{X_{S'}}\otimes \Theta(h)^{\rk(\ind_{S'})}, \arrow[ll, "\pi'!\circ (j_+')^*\circ (j'_-!)^{-1}"]
    \end{tikzcd}
\end{equation*}
\begin{equation*}
\label{stable envelopes }
    \begin{tikzcd}[column sep=huge]
    \Theta(T^{1/2}F^{\vi'})\otimes \tau^*\mathscr{U}_{(X_{\w''}^{\vi''})^{A}}\arrow[rr, "(j''_-)^{A}!\circ (((j_+'')^{A})^*)^{-1} \circ(\pi'')^{A}!)^{-1}"]\arrow[d, "\Stab{C''}{F^{\vi''}}"]& & \Theta(T^{1/2}F_{S''})\otimes \tau^*\mathscr{U}_{X_{S''}^{A}} \arrow[d, "\lambda''^*\Stab{C''}{F_{S''}}"]\\
    \Theta(T^{1/2}X_{\w''}^{\vi''})\otimes \mathscr{U}_{X_{\w''}^{\vi''}}\otimes \Theta(h)^{\rk(\ind^{\vi''})}
    & & \Theta(T^{1/2}X)\otimes \mathscr{U}_{X_{S''}}\otimes  \Theta(h)^{\rk(\ind_{S''})}.\arrow[ll, "\pi''!\circ (j_+'')^*\circ (j''_-!)^{-1}"]
    \end{tikzcd}
\end{equation*}
Let us define
\[
    \Stab{C'}{F_{S'}}\boxtimes \Stab{C''}{F_{S''}}:=  (\phi\times \psi_S)^*\left( pr'^*\Stab{C'}{F_{S'}}\otimes pr''^*\Stab{C''}{F_{S''}}\right)
\]
and notice that 
\begin{align*}
    \lambda^* \Stab{C'}{F_{S'}}\boxtimes \Stab{C''}{F_{S''}}=&
    \lambda^*\circ  (\phi\times \psi_S)^*\left( pr'^*\Stab{C'}{F_{S'}}\otimes pr''^*\Stab{C''}{F_{S''}}\right)
    \\
    =&(\phi\times \psi)^*\circ (id \times \lambda'\times \lambda'')*\left( pr'^*\Stab{C'}{F_{S'}}\otimes pr''^*\Stab{C''}{F_{S''}}\right).
\end{align*}
Taking the exterior product of the two diagrams, pulling back by $\phi\times \psi$ and lifting the abelianized stable envelopes
as discussed in Remark \ref{remark existence lifts}
we get the following commuting square:
\begin{equation*}
\label{stable envelopes }
    \begin{tikzcd}[column sep=normal]
    \Theta(T^{1/2}F^{\vi'}+T^{1/2}F^{\vi''})\otimes \tau^*\mathscr{U}_{(X_{\w}^{\vi',\vi''})^{A}}\arrow[r]\arrow[d, swap, "(\Stab{C'}{F^{\vi'}}\boxtimes \Stab{C''}{F^{\vi''}})^{lift}"]& \Theta(T^{1/2}F_{S'}+T^{1/2}F_{S''})\otimes \tau^*\mathscr{U}_{X_{S}^{A}}\arrow[d, swap, "\lambda^*(\Stab{C'}{F_{S'}}\boxtimes \Stab{C''}{F_{S''}})^{lift}"]\\
    \Theta(T^{1/2}X_\w^{\vi',\vi''}-\mathscr{N}_{\hat j})\otimes \mathscr{U}_{X_\w^{\vi',\vi''}}\otimes \Theta(h)^{\rk(\ind^{\vi',\vi''})}
    & \Theta(T^{1/2}X_S-\mathscr{N}_{\hat j_S})\otimes \mathscr{U}_{X_{S}}\otimes \Theta(h)^{\rk(\ind_S)} \arrow[l]
    \end{tikzcd}
\end{equation*}
where the upper horizontal map is 
\[
(j_-'\times j_-'')^{A}!\circ (((j_+'\times j_+'')^{A})^*)^{-1} \circ((\pi'\times \pi'')^{A}!)^{-1},
\]
the lower one is 
\[
(\pi'\times\pi'')!\circ (j_+'\times j_+'')^*\circ ((j_-'\times j''_-)!)^{-1}
\]
and $\mathcal{N}_{\hat j_S}$ is the normal bundle of the embedding $\hat j_S$ fitting in the following diagram
\begin{equation*}
  \begin{tikzcd}
  X_{S'}\times X_{S''} \arrow[d, swap, bend right, hookrightarrow] \arrow[r,"j_{S}", hookrightarrow] & X_S\\
\Att{\overline{C}}{X_{S'}\times X_{S''}}\arrow[ur, swap, "\hat j_S", hookrightarrow]\arrow[u, bend right, swap, "\rho_S"]  &
\end{tikzcd}  
\end{equation*}
In particular, the previous diagram can serve to define a lift $(\Stab{C'}{F^{\vi'}}\boxtimes \Stab{C''}{F^{\vi''}})^{lift}$.
Notice that from Lemma \ref{lemma normal bundles} and Remark \ref{remark extension bundles} it follows that the classes of $\mathcal{N}_{\hat j}$ and $\mathcal{N}_{\hat j_S}$ are  $\mathcal{M}(\aV',\aV'', \aW', \aW'')$ and $\mathcal{M}(\aV'_S,\aV''_S, \aW'_S, \aW''_S)$ respectively, where $\aV'_S$, $\aV''_S$, $\aW'_S$ and $\aW''_S$ are the tautological bundles on $X_S$ defined as in Section \ref{Tautological bundles}.

Now consider the following extension of the previous diagram:
\begin{equation*}
\label{stable envelopes }
    \begin{tikzcd}[column sep=normal]
    \Theta(T^{1/2}F^{\vi'}+T^{1/2}F^{\vi''})\otimes \tau^*\mathscr{U}_{(X_\w^{\vi',\vi''})^{A}}\arrow[r]\arrow[d, swap, "(\Stab{C'}{F^{\vi'}}\boxtimes \Stab{C''}{F^{\vi''}})^{lift}"]& \Theta(T^{1/2}F_{S'}+T^{1/2}F_{S''})\otimes \tau^*\mathscr{U}_{X_{S}^A}\arrow[d, swap,  "\lambda^*(\Stab{C'}{F_{S'}}\boxtimes \Stab{C''}{F_{S''}})^{lift}"]\\
    \Theta(T^{1/2}X_\w^{\vi',\vi''}-\mathscr{N}_{\hat j})\otimes \mathscr{U}_{X_\w^{\vi',\vi''}}\otimes \Theta(h)^{\rk(\ind^{\vi',\vi''})} 
    \arrow[d, swap, "\Theta(\mathcal{N}_{\hat j})"]& \Theta(T^{1/2}X_S-\mathscr{N}_{\hat j_S})\otimes \mathscr{U}_{X_{S}}\otimes \Theta(h)^{\rk(\ind_S)}  \arrow[l]\arrow[d, swap, "\Theta(\mathcal{N}_{\hat j_S})"]\\
    \Theta(T^{1/2}X_\w^{\vi',\vi''})\otimes \mathscr{U}_{X_\w^{\vi',\vi''}}\otimes\Theta(h)^{\rk(\ind^{\vi',\vi''})} 
    & \Theta(T^{1/2}X_S)\otimes \mathscr{U}_{X_{S}}\otimes\Theta(h)^{\rk(\ind_S)}. \arrow[l, "\pi!\circ (j_+)^*\circ ((j_-)!)^{-1}"]
    \end{tikzcd}
\end{equation*}
We claim that the lower square commutes: $\Theta(\mathcal{N}_{\hat j})$ is multiplication by $\Theta(\mathscr{M}(\aV',\aV'',\aW',\aW''))$ and similarly $\Theta(\mathcal{N}_{\hat j_S})$ is multiplication by $\Theta(\mathscr{M}(\aV',\aV'',\aW',\aW''))$. Moreover, we have that 
\[
(j_+'\times j_+'')^*\circ(j_-'\times j_-'')^*\circ (j)^*\mathscr{M}(\aV',\aV'',\aW',\aW'')=(\pi'\times \pi'')^*\circ(j_S)^*\mathscr{M}(\aV'_S,\aV''_S,\aW'_S,\aW''_S),
\]
so the claim follows. As a consequence, the outer frame in the previous diagram commutes. But the horizontal maps are exactly the horizontal maps in \eqref{diagram abelianization pseudoabelianization}, so in order to prove the theorem it suffices to check that the right vertical map is the stable envelope of the abelianization $X_S$.
\newline
\textbf{Step 3: }As observed before, we are left to show that 
\begin{equation}
\label{last thing to show}
    \Theta(\mathcal{N}_{\hat j_S})\circ (\Stab{C'}{F_{S'}}\boxtimes \Stab{C''}{F_{S''}})^{lift}=\Stab{C}{F_S}.
\end{equation}
This statement is easier to prove because stable envelopes of the the hypertoric variety $X_S$ have an explicit description given in \cite[Section 4.1]{aganagic2016elliptic}. Let us briefly recall this description here: first of all, there exists en extension $i:{A}\hookrightarrow \widetilde A$, where $\widetilde A$ is a torus acting on $X_S$ in such a way that the fixed components are zero-dimensional, i.e. are just points. Moreover, given a chamber $\widetilde{\mathfrak{C}}$ of $\widetilde A$ that restricts to $\mathfrak{C}$ on $A$, the following equation holds:
\[
\text{Stab}^X_{\mathfrak{C}}(F_S)=\text{Stab}_{\widetilde{\mathfrak{C}}}^{X_S}\circ\left(\text{Stab}^{F_S}_{\widetilde{\mathfrak{C}}/\mathfrak{C}}\right)^{-1},
\]
where $\text{Stab}^{F_S}_{\widetilde{\mathfrak{C}}/\mathfrak{C}}$ is the stable envelope of the $\widetilde A/A$-variety $F_S$ and $\widetilde{\mathfrak{C}}/\mathfrak{C}$ denotes the chamber induced by the quotient map $\widetilde A\to\widetilde A/A$.
In our setting, we have that $F_S=F_{S'}\times F_{S''}$, so we deduce that 
\[
\text{Stab}^{F_S}_{\widetilde{\mathfrak{C}}/\mathfrak{C}}=\text{Stab}^{F_{S'}}_{\widetilde{\mathfrak{C}}'/\mathfrak{C}'}\boxtimes \text{Stab}^{F_{S''}}_{\widetilde{\mathfrak{C}}''/\mathfrak{C}''}.
\]
On the other hand
\[
\left(\text{Stab}^{X_{S'}}_{{\mathfrak{C}'}}\boxtimes \text{Stab}^{X_{S''}}_{{\mathfrak{C}''}}\right)^{lift}=\left(\text{Stab}^{X_{S'}}_{\widetilde{\mathfrak{C}}'}\boxtimes \text{Stab}^{X_{S''}}_{\widetilde{\mathfrak{C}}''}\right)^{lift}\circ\left( \text{Stab}^{F_{S'}}_{\widetilde{\mathfrak{C}}'/\mathfrak{C}'}\boxtimes \text{Stab}^{F_{S''}}_{\widetilde{\mathfrak{C}}''/\mathfrak{C}''}\right)^{-1}.
\]
Therefore, to check \eqref{last thing to show} it suffices to show that there exist a chamber $\widetilde{\mathfrak{C}}$ restricting to $\mathfrak{C}$ on $A$ such that the equation
\begin{equation}
\label{last euation main proof}
    \Theta(\mathcal{N}_{\hat j_S})\left(\text{Stab}^{X_{S'}}_{\mathfrak{\widetilde{C}}'}(F')\boxtimes \text{Stab}^{X_{S''}}_{\mathfrak{\widetilde{C}}''}(F'')\right)^{lift}=\text{Stab}^{X_S}_{\mathfrak{\widetilde{C}}}(F'\times F'')
\end{equation}
holds for all $\widetilde A$-fixed points $F'\times F'' \in F_{S'}\times F_{S''}$.
As for the chamber, choose $\widetilde{\mathfrak{C}}$ such that $\chi$ is $\widetilde{\mathfrak{C}}$-attracting whenever $i^*\chi$ is $\mathfrak{C}$-attracting for all characters $\chi$ of the normal bundle of a fixed point $F'\times F''$. This in particular implies that the $\widetilde A$ vector bundle $\mathcal{M}(\aV',\aV'', \aW', \aW'')$ has repelling weights for all $F'\times F''\subset F_{S'}\times F_{S''}$.
By the same arguments of the first step of the first proof, one checks that both sides of equation \eqref{last euation main proof} are morphisms between the same line bundles, so by uniqueness of stable envelopes, it suffices to check that they satisfy the two conditions of Definition \ref{defn stable envelopes}.

Stable envelopes of hypertoric varieties acted on by a torus whose fixed components are zero dimensional are explicitly computed off-shell in \cite[Section 4.1 (56)]{aganagic2016elliptic}, and are of the form
\[
\text{Stab}^{X_S}_{\mathfrak{C}}(F)=\Theta(N_{F,<}^{X_S})W_{F},
\]
where $\Theta(N_{F,<}^{X_S})$ is the Thom section of a class $N_{F,<}^{X_S}$ that, restricted to $F$, coincides with the repelling part of the normal bundle $N_{F}^{X_S}$ of $F$, and $W$ is a function of Chern roots of tautological bundles, equivariant variables and dynamical parameters. In particular $\restr{W_{F_S}}{F_S}=1$ and $\restr{W_{F_S}}{F_S'}\neq 0$ for every $F_S'\neq F_S$, so the conditions \ref{Condition 1 stable env} and \ref{condition 2 stable env} of stable envelopes are uniquely determined by $\Theta(N^{X_S}_{F_S,<})$.
By Assumption \ref{Assumption main theorem} and additivity of the Thom class, we have that 
\[
\Theta\left(N^{X_S}_{F'\times F'',<}\right)=\Theta\left(N^{X_{S'}}_{F',<}\right)\Theta\left(N^{X_{S''}}_{F'',<}\right)\Theta(\mathcal{M}(\aV',\aV'',\aW',\aW'')).
\]
On the other hand, $\text{Stab}^{X_{S'}}_{\widetilde{\mathfrak{C}'}}(F')=\Theta\left(N^{X_{S'}}_{F',<}\right)W_{F'}$ and $\text{Stab}^{X_{S''}}_{\widetilde{\mathfrak{C}''}}(F'')=\Theta\left(N^{X_{S''}}_{F'',<}\right)W_{F''}$, so the left hand side of equation \eqref{last euation main proof} is the section
\[
\Theta\left(N^{X_S}_{F'\times F'',<}\right)=\Theta\left(N^{X_{S'}}_{F',<}\right)\Theta\left(N^{X_{S''}}_{F'',<}\right)\Theta(\mathcal{M}(\aV',\aV'',\aW',\aW''))W_{F'}W_{F''}.
\]
But this implies that both sides of equation \eqref{last euation main proof} satisfy Definition \ref{defn stable envelopes}, so by uniqueness of stable envelopes they must be the same section. 
This concludes the proof.
\end{proof}

\newpage
\printbibliography

\Addresses

\end{document}